\documentclass[a4paper,11pt,reqno]{amsart}
\usepackage[foot]{amsaddr}

\usepackage{amssymb}
\usepackage{epsfig}
\usepackage{amsfonts}
\usepackage{amsmath}
\usepackage{euscript}
\usepackage{amscd}
\usepackage{amsthm}
\usepackage{enumitem}
\DeclareMathAlphabet{\mathpzc}{OT1}{pzc}{m}{it}
\usepackage{enumitem}
\usepackage{color}
\usepackage{mathtools}
\usepackage[pagebackref, hypertexnames=false, colorlinks, citecolor=red, linkcolor=blue, urlcolor=red]{hyperref}

\usepackage{marginnote}

\usepackage[normalem]{ulem}

\newcommand{\marginextend}[1]{ \addtolength{\oddsidemargin}{-#1}  \addtolength{\evensidemargin}{-#1}
  \addtolength{\textwidth}{#1}\addtolength{\textwidth}{#1}}
\newcommand{\updownextend}[1]{ \addtolength{\topmargin}{-#1}  \addtolength{\textheight}{#1}
\addtolength{\textheight}{#1}}
\marginextend{1cm}
\updownextend{0cm}
\allowdisplaybreaks[4]

\usepackage{pst-node}
\usepackage{tikz-cd}
\usepackage{mathrsfs}
\DeclareFontFamily{OT1}{pzc}{}
\DeclareFontShape{OT1}{pzc}{m}{it}{<-> s * [1.10] pzcmi7t}{}
\DeclareMathAlphabet{\mathpzc}{OT1}{pzc}{m}{it}

\DeclareSymbolFont{SY}{U}{psy}{m}{n}
\DeclareMathSymbol{\emptyset}{\mathord}{SY}{'306}

\theoremstyle{plain}

\newtheorem{thm}{Theorem}[section]
\newtheorem*{thm*}{Theorem}
\newtheorem{cor}[thm]{Corollary}
\newtheorem{lem}[thm]{Lemma}
\newtheorem{prop}[thm]{Proposition}
\newtheorem{defn}[thm]{Definition}
\newtheorem{rem}[thm]{Remark}

\newtheoremstyle{named}{}{}{\itshape}{}{\bfseries}{.}{.5em}{#1 \thmnote{#3}}
\theoremstyle{named}

\numberwithin{equation}{section}

\def\C{{\mathbb C}}

\def\norm#1{\left\|{#1}\right\|}

\def\ov{\overline}

\def\m{\mathcal}
\def\mb{\mathbb}

\def\w{\widehat }

\def\beq{\begin{eqnarray}}
\def\eeq{\end{eqnarray}}
\def\beqa{\begin{eqnarray*}}
\def\eeqa{\end{eqnarray*}}

\def\del{\partial}

\def\ov{\overline}
\def\bl{\boldsymbol}

\def\i{\prime}

\def\sgn{{\rm sgn}}
\def\bl{\boldsymbol}
\def\deg{{\rm deg}}
\newcommand{\overbar}[1]{\mkern 1.5mu\overline{\mkern-1.5mu#1\mkern-1.5mu}\mkern 1.5mu}

\newcommand{\be}{\begin{equation}}
\newcommand{\ee}{\end{equation}}
\newcommand{\bea}{\begin{eqnarray}}
\newcommand{\eea}{\end{eqnarray}}
\newcommand{\Bea}{\begin{eqnarray*}}
\newcommand{\Eea}{\end{eqnarray*}}
\newcommand{\inner}[2]{\langle #1,#2 \rangle }%

\newcommand{\innera}[2]{\langle #1,#2 \rangle_\alpha }

\newcounter{cnt1}
\newcounter{cnt2}
\newcounter{cnt3}
\newcommand{\blr}{\begin{list}{$($\roman{cnt1}$)$}
 {\usecounter{cnt1} \setlength{\topsep}{0pt}
 \setlength{\itemsep}{0pt}}}
\newcommand{\bla}{\begin{list}{$($\alph{cnt2}$)$}
 {\usecounter{cnt2} \setlength{\topsep}{0pt}
 \setlength{\itemsep}{0pt}}}
\newcommand{\bln}{\begin{list}{$($\arabic{cnt3}$)$}
 {\usecounter{cnt3} \setlength{\topsep}{0pt}
 \setlength{\itemsep}{0pt}}}
\newcommand{\el}{\end{list}}

\DeclareMathOperator  {\tr} {tr}
\DeclareMathOperator  {\adj}{adj}

\title[Toeplitz operators on the Bergman space]{Toeplitz operators on the weighted Bergman spaces of quotient domains}

\author[Ghosh]{Gargi Ghosh}
\author[Narayanan]{E. K. Narayanan}
\address[Ghosh]{Jagiellonian University, Faculty of Mathematics and Information Technologies, 30-348 Krakow, Poland}
\address[Narayanan]{Indian Institute of Science, Bangalore, 560012, India}
\email[Ghosh]{gargi.ghosh@im.uj.edu.pl}
\email[Narayanan]{naru@iisc.ac.in}

\thanks{The first named author's research is part of the project No. 2022/45/P/ST1/01028 co-funded by the
National Science Centre and the European Union Framework Programme for Research and Innovation Horizon 2020 under the Marie Sklodowska-Curie grant agreement No.945339. The first named author was also partially supported by CV Raman Fellowship and postdocotoral fellowship from Silesian University in Opava under GA CR grant no. 21-27941S. The second named author was supported by the grant MTR/2018/000501 from SERB, India.}
\begin{document}

\maketitle

\begin{abstract}

 Let $G$ be a finite pseudoreflection group and $\Omega\subseteq \mathbb C^d$ be a bounded domain which is a $G$-space. We establish identities involving Toeplitz operators on the weighted Bergman spaces of $\Omega$ and $\Omega/G$ using invariant theory and representation theory of $G.$ This, in turn, provides techniques to study algebraic properties of Toeplitz operators on the weighted Bergman space on $\Omega/G.$ We specialize on the generalized zero-product problem and characterization of commuting pairs of Toeplitz operators.  As a consequence, more intricate results on Toeplitz operators on the weighted Bergman spaces on some specific quotient domains (namely symmetrized polydisc, monomial polyhedron, Rudin's domain) have been obtained.
\end{abstract} 

\section{Introduction}

Let $\Omega$ be a bounded domain in $\mb C^d$ and $\omega : \Omega \to (0,\infty)$ be a continuous function. The weighted Bergman space $\mb A^2_\omega(\Omega)$ is the closed subspace consisting of holomorphic functions in $L^2_\omega(\Omega).$  For $u \in L^\infty(\Omega),$ Toeplitz operator $T_u: \mb A^2_\omega(\Omega) \to \mb A^2_\omega(\Omega)$ is defined by $$(T_u f) = P_\omega(uf),$$ where  $P_\omega : L^2_\omega(\Omega) \to \mb A^2_\omega(\Omega)$ is the orthogonal projection. Suppose that $G$ is a finite group acting on $\Omega.$ Then $\Omega$ is said to be a $G$-invariant domain or a $G$-space. In this article, we study algebraic properties of Toeplitz operators on the weighted Bergman space of $\Omega/G$ for a $G$-invariant domain $\Omega.$ 
The study of Toeplitz operators on function spaces has a long history. Starting from the seminal work of Brown and Halmos \cite{MR160136}, it has attracted a lot of attention.  Brown--Halmos proved various algebraic properties (such as generalized zero-product theorem, characterization of commuting pairs) of Toeplitz operators on the Hardy space $H^2(\mathbb D)$ on the unit disc $\mathbb D.$ Later, Halmos posed the zero-product problem for a finite product of Toeplitz operators and it was solved on $H^2(\mathbb D)$ in full generality in \cite{MR2527321}. Interestingly, even the zero-product problem, in its full generality, is still open for the Bergman space $\mb A^2(\mb D)$. Important result due to Ahern and Cu{c}kovi\'{c} \cite{MR1867348} proved that the answer is affirmative for $\mb A^2(\mb D)$ under additional conditions on the functions. An ingenious proof by Ahern improved this result, by studying the range of the Berezin transform on the Bergman space on $\mathbb D$ \cite{MR2085115}, see also \cite{MR3747963}. There have been several attempts to extend the above results to the higher dimensional situations, in particular, to the Hardy space and Bergman space over the unit ball and the polydisc in $\mathbb C^d,$ see \cite{MR2214579}, \cite{MR2294274}, \cite{MR2262785}, \cite{MR1624149}, \cite{MR1443898} and the references therein. As for the commutativity, Axler and Cu{c}kovi\'{c} gave necessary and sufficient conditions on the bounded harmonic symbols $u$ and $v$ such that Toeplitz operators $T_u$ and $T_v$ on $\mb A^2(\mb D)$ commute \cite{MR1079815}. Since then various attempts have been made to generalize this problem in polydisc and unit ball in $\mb C^d$ \cite{MR1443898, MR1624149, MR2031039}.  The study of algebraic properties of Toeplitz operators on the Hardy spaces on the quotient domains emerged very recently \cite{MR4266152, MR4244837}. However, the
analogous questions in the Bergman spaces of quotient domains turn out to be much more difficult and of worth to study. 

To achieve our goal, we first observe that it is not obvious that $\Omega/G$ is a domain for a $G$-invariant domain $\Omega$. However, it is known that $\Omega/G$ can be given the structure of a complex analytic space which is biholomorphically equivalent to some domain in $\mb C^d$ whenever $G$ is a finite pseudoreflection group \cite{MR807258} \cite[Subsection 3.1.1]{MR4404033} \cite[Proposition 1]{MR2542964}. Henceforth, we confine our attention to $\Omega/G$ where $G$ is a finite pseudoreflection group. Recall that a \emph{pseudoreflection} on $\C^d$ is a linear homomorphism $\sigma: \C^d \rightarrow \C^d$ such that $\sigma$ has finite order in $GL(d,\mb C)$ and the rank of $(I_d - \sigma)$ is 1. A group generated by pseudoreflections is called a pseudoreflection group. For example, any finite cyclic group, the permutation group $\mathfrak{S}_d$ on $d$ symbols, the dihedral groups are finite pseudoreflection groups \cite{MR2542964}. A pseudoreflection group $G$ acts on $\mb C^d$ by (right action) $\sigma \cdot \bl z = \sigma^{-1} \bl z$ for $\sigma \in G$ and $\bl z \in \mb C^d.$ The right action on $\mb C^d$ gives rise to the left action of $G$ on the set of all complex-valued functions on $\mb C^d$ as following: \bea\label{action}\sigma( f)(\bl z) =  f({\sigma}^{-1}\cdot \bl z), \,\, \text{~for~} \sigma \in G \text{~and~} \bl z \in \mb C^d.\eea  If $\Omega$ is a $G$-space under the above action then the quotient $\Omega / G$ is biholomorphically equivalent to the domain $\bl \theta(\Omega),$ where $\bl \theta: \mb C^d \to \mb C^d$ is a basic polynomial map associated to the finite pseudoreflection group $G$ \cite{MR4404033,MR3133729}. Notice that if $\widetilde{\Omega}$ is a domain such that there exists a proper holomorphic map $\bl f : \Omega \to \widetilde{\Omega}$ with $G$ as the group of deck transformations, then $\widetilde{\Omega}$ is biholomorphic to $\Omega/G$ and $\bl \theta$ is a representative of $\bl f,$ that is, $\bl f = \bl \theta \circ \bl h$ for some biholomorphism $\bl h : \bl \theta(\Omega) \to \widetilde{\Omega}$ \cite[Proposition 2.2]{kag}. Therefore, without loss of generality, we work with the domain $\bl \theta(\Omega)$ instead of $\Omega / G.$  It is important to note that this phenomena is not in general true for any group $G.$ Below are a few examples of well-studied quotient domains. 

\begin{itemize}
\item The symmetrized polydisc is a quotient domain biholomorphic to $\mb D^d/\mathfrak{S}_d$ where $\mathfrak{S}_d$ denotes the permutation group on $d$ symbols \cite{MR3043017}. 

\item The tetrablock is biholomorphic to the quotient domain $\m R_{II}/\mathfrak{S}_2$ where $\m R_{II}$ is the classical Cartan domain of second type \cite{MR3133729}. 

\item Rudin's domains are realized as $\mb B_d / G$ ($\mb B_d$ is the unit ball of $\mb C^d$ with respect to $\ell^2$-norm) for a finite pseudoreflection group $G$ \cite{MR667790}.

\item A monomial polyhedron is a quotient domain $\Omega / G$ where $\Omega \subseteq \mb D^d$ and  $G$ is a finite abelian group \cite{bender2020lpregularity}.
\end{itemize} 


The study of Toeplitz operators on the Bergman space largely depends on the geometry and function theory of the domain. In most of the cases,  the quotient domains $\Omega/G$  behave differently from the domain $\Omega$ and the geometry of such domains can be complicated. For example, though both of the polydisc $\mb D^d$ and the classical Cartan domain of second type $\m R_{II}$ are homogeneous, the quotient domains $\mb D^d/\mathfrak{S}_d$ and $\m R_{II}/\mathfrak{S}_2$ are not homogeneous \cite[p. 265]{MR2108554}  \cite[p. 764, Corollary 3.2]{MR2418303}.
On the other hand, characterization of commuting Toeplitz operators on $\mb A^2(\Omega)$ (for $\Omega=\mb D, \mb D^d$ and $\mb B_d)$ depends on the homogeneity of the domains to a great extent \cite{MR1443898,MR2031039,MR961609,MR993443}. We emphasize that our method relies only on invariant theory and representation theory of the group $G,$ which provides us a greater freedom to work with domains $\Omega/G$ without making any appeal to their geometry. 
We now briefly describe the main results of this paper.
\subsection{Generalized zero-product problem}

Let $\omega : \Omega \to (0,\infty)$ be a $G$-invariant continuous function. For each one-dimensional representation $\varrho$ of $G,$ we define the relative invariant subspace of the weighted Bergman space $\mb A_\omega^2(\Omega)$ by
\Bea
R^G_{\varrho}(\mb A_\omega^2(\Omega)) = \{f \in \mb A_\omega^2(\Omega) : \sigma(f)  = \chi_\varrho(\sigma) f ~ {\rm for ~~ all~} \sigma \in G \},
\Eea
where $\chi_\varrho$ denotes the character of $\varrho.$  A characterization of the relative invariant subspace is that every $f \in R^G_{\varrho}(\mb A_\omega^2(\Omega))$ is divisible by some polynomial $\ell_\varrho$ and the quotient is in the ring of $G$-invariant holomorphic functions on $\Omega$ \cite[p. 8, Lemma 2.8]{kag}. An explicit expression for $\ell_\varrho$ has been obtained from the representation $\varrho$ \cite[p. 139, Theorem 3.1]{MR460484} (cf. Lemma \ref{gencz}). For each one-dimensional representation $\varrho$ of $G,$ we set \bea\label{wei}\omega_\varrho(\bl \theta(\bl z)) = \frac{|\ell_\varrho(\bl z)|^2}{|J_{\bl \theta}(\bl z)|^2} \omega(\bl z),\eea where $J_{\bl \theta}$ is the determinant of the complex jacobian of the basic polynomial map $\bl \theta$ associated to the group $G.$ Then $R^G_{\varrho}(\mb A_\omega^2(\Omega))$ is isometrically isomorphic to the weighted Bergman space $\mb A_{\omega_\varrho}^2(\bl \theta(\Omega))$ with weight function $\omega_\varrho$ \cite[p. 2, Theorem 1.1]{kag}. 

We also obtain that for a $G$-invariant function $\widetilde{u}$ in $L^\infty(\Omega),$ there exists $u \in L^\infty(\bl \theta(\Omega))$ such that $\widetilde{u} = u \circ \bl \theta$ (cf. Remark \ref{rem1}). 
 \begin{thm}\label{main1}
 Suppose that $G$ is a finite pseudoreflection group, the bounded domain $\Omega \subseteq \mb C^d$ is a $G$-space and $\bl \theta: \Omega \to \bl \theta(\Omega)$ is a basic polynomial map associated to the group $G.$ Let $\widetilde{u},\widetilde{v}$ and $\widetilde{q}$ be $G$-invariant functions in $L^\infty(\Omega)$ such that $\widetilde{u} = u \circ \bl \theta$, $\widetilde{v} = v \circ \bl \theta$ and $\widetilde{q} = q \circ \bl \theta$. 
 \begin{enumerate}
     \item[{\rm 1.}] Suppose that for a one-dimensional representation $\mu$ of $G,$ $T_uT_v=T_q$ on $\mb A^2_{\omega_\mu}(\bl \theta(\Omega)),$ then \begin{enumerate}
    \item[\rm (i)] $T_uT_v=T_q$ on $\mb A^2_{\omega_\varrho}(\bl \theta(\Omega))$ for every one-dimensional representation $\varrho$ of $G,$ and
    \item[\rm (ii)] $T_{\widetilde{u}}T_{\widetilde{v}}=T_{\widetilde{q}}$ on $\mb A_\omega^2(\Omega),$
\end{enumerate}
where the weight function $\omega : \Omega \to (0,\infty)$ is $G$-invariant and continuous and $\omega_\varrho : \bl \theta(\Omega) \to (0,\infty)$ is as defined in Equation \eqref{wei} for every one-dimensional representation $\varrho$.
\item[{\rm 2.}] Conversely, if $T_{\widetilde{u}}T_{\widetilde{v}}=T_{\widetilde{q}}$ on $\mb A_\omega^2(\Omega),$ then $T_uT_v=T_q$ on $\mb A^2_{\omega_\varrho}(\bl \theta(\Omega))$ for every one-dimensional representation $\varrho.$
 \end{enumerate}
 
\end{thm}

We provide a number of consequences of Theorem \ref{main1} in Sections \ref{appl} and \ref{genzer}. We briefly describe some of those below.
\begin{enumerate}[leftmargin=*]
    \item[1.]{\sf Zero-product theorem.} 
    Let $h_G^\infty(\bl \theta(\Omega))$ denote the set of all bounded $G$-pluriharmonic functions on $\bl \theta(\Omega)$ (cf. Definition \eqref{ph1}) and $\del_S\Omega$ denote the Shilov boundary of $\Omega.$ For the weight function $\omega_\varrho$ as in Equation \eqref{wei} with $\omega \equiv 1,$ we have the following as an application of Theorem \ref{main1}:
    \begin{thm}\label{result}
Let $\Omega = \mb D^d$ or $\mb B_d$ and $G$ be a finite pseudoreflection group with a basic polynomial map $\bl \theta.$ Suppose that $u, v \in h_G^\infty(\bl \theta(\Omega))$ are continuous on $\bl \theta(\Omega) \cup W$ for some relatively open subset $W \text{~of~} \del_S\bl \theta(\Omega).$ If $T_uT_v=0$ on $\mb A_{\omega_\varrho}^2(\bl \theta(\Omega))$ for a one-dimensional representation $\varrho $ of $ G,$ then either $u=0$ or $v=0.$
\end{thm}
Theorem \ref{result} provides explicit sufficient conditions for zero-product problem of Toeplitz operators on the weighted Bergman spaces of quotient domains such as the symmetrized polydisc, Rudin's domain, monomial polyhedron, see section \ref{appl}. In particular, the permutation group $\mathfrak{S}_d$ acts on the polydisc $\mb D^d$ and the symmetrization map $\bl s$ (cf. Equation \eqref{sym}) is a basic polynomial map associated to $\mathfrak{S}_d.$ The domain $\mb G_d:= \bl s(\mb D^d)$ is said to be the symmetrized polydisc. For  the sign representation of $\mathfrak{S}_d,$ $\mb A^2_{\omega_{\rm sign}}(\mb G_d)$ reduces to the Bergman space $\mb A^2(\mb G_d).$ Then a straightforward application of Theorem \ref{result} yields the following: 
\begin{cor} 
Suppose that $u, v \in h_{\mathfrak{S}_d}^\infty(\mb G_d)$ are continuous on $\mb G_d \cup W$ for some relatively open subset $W \text{~of~} \bl s(\mb T^d).$ If $T_uT_v=0$ on $\mb A^2(\mb G_d)$ then either $u=0$ or $v=0.$
\end{cor}
\item[2.]{\sf Generalized zero-product theorem on the weighted Bergman space on $\mb G_d.$} For $\alpha > -1,$ the continuous function $\omega_\alpha: \mb D \to (0,\infty)$ is defined by $\omega_\alpha(z) = (\alpha +1 ) (1-|z|^2)^\alpha.$ A similar result as Ahern and Cu{c}kovi\'{c}'s on generalized zero-product problem of Toeplitz operators on $\mb A^2(\mb D)$ \cite{MR1867348} is established for Toeplitz operators on $\mb A^2_{\omega_\alpha}(\mb D)$ in \cite{MR4295248}, under additional conditions on the symbols (cf. Proposition \ref{exone}). 
We extend that result for Toeplitz operators on $\mb A^2_{\omega_{\bl \alpha}}(\mb D^d),$ where $\omega_{\bl \alpha}: \mb D^d \to (0,\infty)$ is defined by $\omega_{\bl \alpha}(\bl z) = \prod_{i=1}^d(\alpha_i +1 ) (1-|z_i|^2)^{\alpha_i}$ for $\bl \alpha = (\alpha_1,\ldots,\alpha_d), \,\, \alpha_i \in \mb N\cup\{0\}.$ Let $\widetilde{\Delta}_{i}=(1-|z_i|^2)^2\frac{\del^2}{\del z_i\del\ov{z}_i}.$
\begin{thm}\label{expoly}
Let $f, g \in L^\infty(\mb D^d)$ be pluriharmonic symbols and $h \in L^\infty(\mb D^d)$ such that $\widetilde{\Delta}_i^{n_i}h \in L^1(\mb D^d, \omega_{\bl \alpha} dV)$ for $i=1,\ldots,d$ and $n_i=0,1,\ldots,\alpha_i.$ If $T_fT_g = T_h$ on $\mb A^2_{\omega_{\bl \alpha}}(\mb D^d),$ then $h=fg$ and either $\del_jf = 0$ or $\del_j\ov{g} =0$ for each $j=1,\ldots,d.$
\end{thm}
For $\bl \alpha= (\alpha,\ldots,\alpha),\,\alpha \in \mb N\cup\{0\},$ the weight function $\omega_{\bl \alpha}$ is $\mathfrak{S}_d$-invariant. Then there exists $\widetilde{\omega}_{\bl \alpha} \in L^\infty(\mb G_d)$ such that $\omega_{\bl \alpha} = \widetilde{\omega}_{\bl \alpha} \circ \bl s,$ where $\bl s$ is the symmetrization map as defined in Equation \eqref{sym}. Combining Theorem \ref{expoly} and Theorem \ref{main1}, we have the following result.
\begin{thm}\label{forsympoly}
Let $u,v \in L^\infty(\mb G_d)$ be $\mathfrak{S}_d$-pluriharmonic function and $q \in L^\infty(\mb G_d)$ such that $\widetilde{\Delta}_i^{t}(q \circ \bl s) \in L^1(\mb D^d, \omega_{\bl \alpha} dV)$ for $i=1,\ldots,d$ and $t=0,1,\ldots,\alpha.$ If $T_uT_v = T_q$ on  $\mb{A}_{\widetilde{\omega}_{\bl \alpha}} (\mb G_d),$ then $q=uv$ and either $\del_j\widetilde{u} = 0$ or $\del_j\ov{\widetilde{v}} =0$ for each $j=1,\ldots,d,$ where $\widetilde{u} = u \circ \bl s,\widetilde{v}=v \circ \bl s.$
\end{thm}
\end{enumerate}

\subsection{Characterization of commuting pairs of Toeplitz operators}

The following theorem states that a pair of commuting Toeplitz operators on the weighted Bergman space on a quotient domain $\Omega/G$ can be characterized by a pair of commuting Toeplitz operators on $\mb A^2(\Omega)$ whenever $G$ is a finite pseudoreflection group.
\begin{thm}\label{main2}
Suppose that $G$ is a finite pseudoreflection group, the bounded domain $\Omega \subseteq \mb C^d$ is a $G$-space and $\bl \theta: \Omega \to \bl \theta(\Omega)$ is a basic polynomial map associated to the group $G.$ Let $\widetilde{u}$ and $\widetilde{v}$ be $G$-invariant functions in $L^\infty(\Omega)$ such that $\widetilde{u} = u \circ \bl \theta$ and $\widetilde{v} = v \circ \bl \theta$. 
\begin{enumerate}
    \item[{\rm 1.}] If for a one-dimensional representation $\mu$ of $G,$ $T_uT_v=T_v T_u$ on $\mb A^2_{\omega_\mu}(\bl \theta(\Omega)),$ then
\begin{enumerate}
    \item[\rm (i)] $T_uT_v=T_v T_u$ on $\mb A^2_{\omega_\varrho}(\bl \theta(\Omega))$ for every one-dimensional representation $\varrho$ and
    \item[\rm (ii)] $T_{\widetilde{u}}T_{\widetilde{v}}=T_{\widetilde{v}}T_{\widetilde{u}}$ on $\mb A_\omega^2(\Omega),$
\end{enumerate}
where the weight function $\omega : \Omega \to (0,\infty)$ is $G$-invariant and continuous and $\omega_\varrho : \bl \theta(\Omega) \to (0,\infty)$ is as defined in Equation \eqref{wei} for every one-dimensional representation $\varrho$.
    \item[{\rm 2.}] Conversely, if $T_{\widetilde{u}}T_{\widetilde{v}}=T_{\widetilde{v}}T_{\widetilde{u}}$ on $\mb A^2_\omega(\Omega),$ then $T_uT_v=T_vT_u$ on $\mb A^2_{\omega_\varrho}(\bl \theta(\Omega))$ for every one-dimensional representation $\varrho.$
\end{enumerate}\end{thm}
Using the above results, we provide  characterizations of a pair of commuting Toeplitz operators on the quotient domains such as symmetrized polydisc and Rudin's domains in Section \ref{appl}.

\section{Preliminaries}\label{section 2}
We begin by recalling a number of useful definitions and standard results about pseudoreflection groups. 
\subsection{Chevalley-Shephard-Todd theorem}
Let $G$ be a finite pseudoreflection group acting on the set of complex-valued functions on $\mb C^d$ by the action defined in Equation \eqref{action}. A function $f$ is said to be $G$-invariant if $\sigma(f)=f$ for all $\sigma \in G.$ We denote the ring of all complex polynomials in $d$-variables by $\mb C[z_1,\ldots,z_d]$. The set of all $G$-invariant polynomials, denoted by $\mb C[z_1,\ldots,z_d]^G$, forms a subring and coincides with the relative invariant subspace $R^G_{tr}(\mb C[z_1,\ldots,z_d])$ associated to trivial representation of $G$. Chevalley, Shephard and Todd characterized finite pseudoreflection groups in the following theorem. 
\begin{thm*}[Chevalley-Shephard-Todd theorem]\cite[ p. 112, Theorem 3]{MR1890629}\label{A}
The invariant ring $\C[z_1,\ldots,z_d]^G$ is equal to $\C[\theta_1,\ldots,\theta_d]$, where $\theta_i$'s are algebraically independent
homogeneous polynomials if and only if $G$ is a finite pseudoreflection group. 
\end{thm*}
We abbreviate it as CST theorem for further references. The collection of homogeneous polynomials $\{\theta_i\}_{i=1}^d$ is called a homogeneous system of parameters (hsop) or a set of basic polynomials associated to the pseudoreflection group $G$. 
The map ${\bl\theta}: \C^d \rightarrow \C^d$, defined by
\bea\label{theta}
{\bl\theta}(\bl z) = \big(\theta_1(\bl z),\ldots,\theta_d(\bl z)\big),\,\,\bl z\in\C^d
\eea is called a basic polynomial map associated to the group $G.$ 
\begin{prop}\label{properholo}
Let $G$ be a finite pseudoreflection group and $\Omega$ be a $G$-space. For a basic polynomial map $\bl \theta: \mb C^d \to \mb C^d$ associated to the group $G,$
\begin{enumerate}
    \item[{\rm (i)}] $\bl \theta(\Omega)$ is a domain and
    \item[{\rm (ii)}] $\bl \theta : \Omega \to \bl \theta(\Omega)$ is a proper map.
    \item[{\rm (iii)}] The quotient $\Omega / G$ is biholomorphically equivalent to the domain $\bl \theta(\Omega)$
\end{enumerate}
\end{prop}
Proof of (i) and (ii) can be found in \cite[Proposition 1, p.556]{MR3133729}. The remaining part of the Proposition follows from \cite[Proposition 1]{MR807258}, see also \cite[Subsection 3.1.1]{MR4404033}, \cite{MR2542964}.
\begin{rem}\label{rem1}
We note a few relevant properties of a basic polynomial map associated to $G$.
\begin{enumerate}
    \item[$1.$] Although a set of basic polynomials associated to $G$ is not unique but the degrees of $\theta_i$'s are unique for $G$ up to order. 
    \item[$2.$]Let $\bl \theta^\prime : \mb C^d \to \mb C^d$ be another basic polynomial map of $G.$ Then $\bl \theta^\prime(\Omega)$ is biholomorphically equivalent to $\bl \theta(\Omega)$ \cite[p. 5, Proposition 2.2]{kag}. Therefore, the notion of a basic polynomial map can be used unambiguously in the sequel. 
    \item[$3.$] Clearly, any function on $\bl \theta(\Omega)$ can be associated to a $G$-invariant function on $\Omega.$ Conversely, any $G$-invariant function $u$ on $\Omega$ can be written as $u=\widehat{u} \circ \bl \theta$ for a function $\widehat{u}$ on $\bl \theta(\Omega).$ Let $\mathfrak q : \Omega \to \Omega/G$ be the quotient map. Since the function $u$ is $G$-invariant, so $u=u_1\circ\mathfrak q$ for some function $u_1$ defined on $\Omega/G$. It is known that $\bl \theta = \mathfrak h \circ \mathfrak q$ for a biholomorphic map $\mathfrak h : \Omega/G \to \bl \theta(\Omega)$ \cite[p. 8, Proposition 3.4]{MR4404033}. Then $u$ can be written as $u=\widehat{u}\circ \mathfrak h \circ \mathfrak q=\widehat{u} \circ \bl \theta.$ Clearly, if $u$ is in $L^\infty(\Omega),$ then $\widehat{u} \in L^\infty(\bl \theta(\Omega)).$
\end{enumerate}
\end{rem}

\subsection{One-dimensional representations of pseudoreflection groups}
Since the one-dimensional representations of a finite pseudoreflection group $G$ play an important role in our discussion, we elaborate on some relevant results for the same. We denote the set of equivalence classes of one-dimensional representations of $G$ by $\w{G}_1$.
\begin{defn}
A hyperplane $H$ in $\C^d$ is called reflecting if there exists a pseudoreflection in $G$ acting trivially
on $H$.
\end{defn}
  For a pseudoreflection $\sigma \in G,$ define $H_{\sigma} := \ker(I_d - \sigma).$ By definition, the subspace $H_{\sigma}$ has dimension $d-1$. Clearly, $\sigma$ fixes the hyperplane $H_{\sigma}$ pointwise. Hence each $H_\sigma$ is a reflecting hyperplane.  By definition, $H_\sigma$ is the zero set of a non-zero homogeneous linear polynomial $L_\sigma$ on $\C^d$, determined up to a non-zero constant multiple, that is, $H_\sigma = \{\bl z\in\C^d: L_\sigma(\bl z) = 0\}$. Moreover, the elements of $G$ acting trivially on a  reflecting hyperplane forms a cyclic subgroup of $G$. 
  
  Let $H_1,\ldots, H_t$ denote the distinct reflecting hyperplanes associated to the group $G$ and  the corresponding cyclic subgroups are $G_1,\ldots, G_t,$ respectively. Suppose $G_i = \langle a_i \rangle$ and the order of each $a_i$ is $m_i$ for $i=1,\ldots,t.$ For every one-dimensional representation $\varrho,$ there exists a unique $t$-tuple of non-negative integers $(c_1,\ldots,c_t),$ where $c_i$'s are the least non-negative integers that satisfy the following: \bea\label{ci}\varrho(a_i) =\big( \det(a_i)\big)^{c_i}, \,\, i=1,\ldots,t.\eea The $t$-tuple $(c_1,\ldots,c_t)$ solely depends on the representation $\varrho.$ The character of the one-dimensional representation $\varrho,$ $\chi_\varrho : G \to \mb C^*$ coincides with the representation $\varrho.$ 
The set of polynomials relative to the representation $\varrho \in \w{G}_1$ is given by \bea\label{invar} R^G_{\varrho}(\mb C[z_1,\ldots,z_d]) = \{f \in \mb C[z_1,\ldots,z_d] : \sigma(f) = \chi_\varrho(\sigma) f, ~ {\rm for ~~ all~} \sigma \in G\}.\eea  The elements of the subspace $R^G_{\varrho}(\mb C[z_1,\ldots,z_d])$ are said to be  $\varrho$-invariant polynomials. Stanley proves a fundamental property of the elements of $R^G_{\varrho}(\mb C[z_1,\ldots,z_d])$ in \cite[p. 139, Theorem 3.1]{MR460484}.
\begin{lem}\label{gencz}\cite[p. 139, Theorem 3.1]{MR460484}
Suppose that the linear polynomial $\ell_i$ is a defining function of $H_i$ for $i=1,\ldots,t.$ The homogeneous polynomial $\ell_\varrho = \prod_{i=1}^t \ell_i^{c_i }$ forms a basis of the module $R^G_{\varrho}(\mb C[z_1,\ldots,z_d])$ over the ring $\mb C[z_1,\ldots,z_d]^G,$ where $c_i$'s are unique non-negative integers as described in Equation \eqref{ci}.
\end{lem}
We call $\ell_\varrho$ by \emph{generating polynomial} of $R^G_{\varrho}(\mb C[z_1,\ldots,z_d])$ over $\mb C[z_1,\ldots,z_d]^G.$ It follows that $\sigma (\ell_{\varrho}) = \chi_{\varrho}(\sigma) \ell_{\varrho}.$ The sign representation of a finite pseudoreflection group $G,$ $\sgn : G \to \mb C^*,$ defined by \bea\label{sign} \sgn(\sigma) = (\det(\sigma))^{-1},\eea is given by ${\rm sgn} (a_i) = \big(\det(a_i)\big)^{m_i-1}=(\det(a_i))^{-1}$, $i=1,\ldots,t,$ \cite[p. 139, Remark (1)]{MR460484} and it has the following property. 

\begin{cor}\cite[p. 616, Lemma]{MR117285}\label{Jac}
Let $H_1,\ldots, H_t$ denote the distinct reflecting hyperplanes associated to the group $G$ and let $m_1,\ldots, m_t$ be the orders of the corresponding cyclic subgroups $G_1,\ldots, G_t,$ respectively. Then for a non-zero constant $c$, 
\Bea
J_{\bl \theta} (\bl z) = c \prod_{i=1}^t \ell_i^{m_i -1 }(\bl z) = \ell_{\rm sgn}(\bl z),
\Eea
where $J_{\bl \theta}$ is the determinant of the complex jacobian matrix of the basic polynomial map $\bl \theta.$
Consequently, $J_{\bl \theta}$ is a basis of the module $R^G_{{\rm sgn}}(\mb C[z_1,\ldots,z_d])$ over the ring $\mb C[z_1,\ldots,z_d]^G$.
\end{cor}

\section{Decomposition of weighted $L^2$-spaces} In this section, we define the relative invariant subspaces of $L^2_\omega(\Omega)$ and $\mathbb A_\omega^2(\Omega)$ associated to $\varrho \in \widehat{G}_1.$ Then we show that those subspaces can be naturally identified to some weighted $L^2$-spaces and weighted Bergman spaces on $\bl \theta (\Omega),$ respectively. We also include some necessary details of the relative invariant subspaces of $L^2_\omega(\Omega)$ and $\mathbb A_\omega^2(\Omega)$. \subsection{Isotypic decomposition and projection operators} 

Given a continuous $G$-invariant weight function $\omega: \Omega \to (0,\infty),$ $L_\omega^2(\Omega)$ denotes the Hilbert space of Lebesgue measurable functions (equivalence classes of functions) on $\Omega$ which are square integrable with respect to the measure $\omega(\bl z)dV(\bl z),$ where $dV$ is the normalized Lebesgue measure on $\Omega.$ The weighted Bergman space $\mb A^2_\omega(\Omega)$ is the closed subspace consisting of holomorphic functions in $L_\omega^2(\Omega)$. For $\omega \equiv 1,$ $\mb A^2_\omega(\Omega)$ reduces to the Bergman space $\mb A^2(\Omega).$ We consider the natural action (sometimes called regular representation) of $G$ on $L_\omega^2(\Omega),$ given by $\sigma (f) (z) = f(\sigma^{-1} \cdot z).$ This action is a unitary representation of $G$ (as the weight $\omega$ is $G$-invariant) and consequently the space $L^2_\omega(\Omega)$ decomposes into isotypic components.

Now we define the projection operator onto the isotypic component associated to an irreducible representation $\varrho \in \widehat{G}$ in the decomposition of the regular representation on $L_\omega^2(\Omega)$ \cite[p. 24, Theorem 4.1]{MR2553682}. For $\varrho \in \widehat{G},$ the linear operator $\mb P_\varrho:L_\omega^2(\Omega) \to L_\omega^2(\Omega)$ is defined by
\Bea \mb P_\varrho \phi = \frac{\deg(\varrho)}{|G|}\sum_{\sigma \in G} \chi_\varrho(\sigma^{-1}) ~ \phi \circ \sigma^{-1}, \, \, \, \phi \in L^2_\omega(\Omega),\Eea where $\chi_\varrho$ denotes the character of $\varrho$ and $|G|$ denotes the order of the group $G.$
\begin{lem}
For each $\varrho \in \widehat{G}$, the operator $\mb P_\varrho:L_\omega^2(\Omega) \to L_\omega^2(\Omega)$ is an orthogonal projection.
\end{lem}
\begin{proof}
An application of Schur's Lemma implies that $\mb P_\varrho^2 = \mb P_\varrho$ \cite[p. 24, Theorem 4.1]{MR2553682}. 
We now show that $\mb P_\varrho$ is self-adjoint. Using change of variables formula, we get that for all $\phi, \psi \in L^2_\omega(\Omega) \text{~and~} \sigma \in G,$ \bea\label{inv}\inner{\sigma \cdot \phi}{\sigma \cdot \psi} = \inner{\phi}{\psi},\eea where $\inner{\cdot}{\cdot}$ denotes the inner product in $L^2_\omega(\Omega).$ For $\phi, \psi \in L^2_\omega(\Omega),$ we have
\Bea \inner{\mb P_\varrho^* \phi}{\psi} = \inner{\phi}{\mb P_\varrho\psi} &=& \inner{\phi}{\frac{\deg(\varrho)}{|G|}\sum_{\sigma \in G} \chi_\varrho(\sigma^{-1}) ~ \psi \circ \sigma^{-1}}\\ &=& \frac{\deg(\varrho)}{|G|}\sum_{\sigma \in G} \chi_\varrho(\sigma) \inner{\phi}{\psi \circ \sigma^{-1}} \\ &=&  \frac{\deg(\varrho)}{|G|}\sum_{\sigma \in G} \chi_\varrho(\sigma) \inner{\phi \circ \sigma}{\psi} \\&=& \inner{\mb P_\varrho \phi}{\psi}, \Eea
where the penultimate equality follows from Equation \eqref{inv}.
\end{proof}
Therefore, $L_\omega^2(\Omega)$ can be decomposed into an orthogonal direct sum as follows: \bea\label{l2ortho} L^2_\omega(\Omega) = \oplus_{\varrho \in \widehat{G}} \mb P_\varrho (L_\omega^2(\Omega)). \eea
\begin{rem}
The space $\mb A_\omega^2(\Omega)$ is clearly $G$-invariant and hence admits a decomposition as above. 
The linear map $\mb P_\varrho:\mb A_\omega^2(\Omega) \to \mb A_\omega^2(\Omega),$ defined by,
\Bea \mb P_\varrho \phi = \frac{\deg(\varrho)}{|G|}\sum_{\sigma \in G} \chi_\varrho(\sigma^{-1}) ~ \phi \circ \sigma^{-1}, \, \, \, \phi \in \mb A_\omega^2(\Omega) \Eea is the orthogonal projection onto the isotypic component associated to the irreducible representation $\varrho$ in the decomposition of the regular representation on $\mb A_\omega^2(\Omega)$ \cite[p. 24, Theorem 4.1]{MR2553682} \cite[Corollary 4.2]{MR4404033} and
\bea\label{decomp} \mb A^2_\omega(\Omega) = \oplus_{\varrho \in \widehat{G}} \mb P_\varrho (\mb A^2_\omega(\Omega)). \eea
\end{rem}
 For every one-dimensional representation $\varrho,$ we provide a characterization of the subspace $\mb P_\varrho(L^2_\omega(\Omega))$ in next lemma. Generalizing the notion of a relative invariant subspace, defined in Equation \eqref{invar}, we define the relative invariant subspace of $L_\omega^2(\Omega)$ associated to a one-dimensional representation $\varrho$ of $G,$ by
\Bea
R^G_{\varrho}(L^2_\omega(\Omega)) = \{f \in L_\omega^2(\Omega) : \sigma(f)  = \chi_\varrho(\sigma) f ~ {\rm for ~~ all~} \sigma \in G \}.
\Eea
\begin{lem}\label{equal1}
Let $G$ be a finite pseudoreflecion group and $\Omega$ be a $G$-invariant domain in $\C^d.$ Then for every $\varrho \in \widehat{G}_1,$ $\mb P_\varrho(L^2_\omega(\Omega)) = R^G_{\varrho}(L_\omega^2(\Omega)).$ 
\end{lem}
\begin{proof}
For any $\tau \in G,$ $\phi \in \mb P_\varrho(L^2_\omega(\Omega)),$ \Bea \tau(\phi) = \tau(\mb P_\varrho \phi) &=& \frac{1}{|G|} \sum_{\sigma \in G} \chi_\varrho(\sigma^{-1}) ~ \phi \circ \sigma^{-1}\tau^{-1} \\ &=& \frac{1}{|G|} \sum_{\eta \in G} \chi_\varrho(\eta^{-1}\tau ) ~ \phi \circ \eta^{-1} \,\, (\text{~taking~}\eta = \tau\sigma)\\ &=& \chi_\varrho(\tau) ~ \phi. \Eea
Conversely, for $\phi \in L^2_\omega(\Omega)$ with the property $\sigma(\phi) = \chi_\varrho(\sigma) ~ \phi,$ we have \Bea \mb P_\varrho (\phi) = \frac{1}{|G|} \sum_{\sigma \in G}  \chi_\varrho(\sigma^{-1}) ~ \phi \circ \sigma^{-1} &=& \frac{1}{|G|} \sum_{\sigma \in G}  \chi_\varrho(\sigma^{-1}) \chi_\varrho(\sigma) ~ \phi  \\ &=& \frac{1}{|G|} \sum_{\sigma \in G} \phi = \phi. \Eea
\end{proof}
The elements of $\mb P_\varrho(L^2_\omega(\Omega))$ are called $\varrho$-invariant elements of $L^2_\omega(\Omega)$. For the trivial representation, $\tr: G \to \mb C^*$ such that \bea\label{trivial}\tr(\sigma) =1,\,\, \sigma \in G,\eea we have $\mb P_{\tr}(L^2_\omega(\Omega)) = \{ f \in L^2_\omega(\Omega): \sigma(f) = f \text{~ for all ~} \sigma \in G \}.$ We refer to the elements of $\mb P_{\tr}(L_\omega^2(\Omega))$ as $G$-invariant elements of $L^2_\omega(\Omega)$. 

Suppose that $\ell_\varrho  \w{\phi} \in L^2_\omega(\Omega)$ for some $G$-invariant $\w{\phi},$ then $\sigma(\ell_\varrho  \w{\phi}) = \chi_\varrho(\sigma) \ell_\varrho  \w{\phi},$ that is, $\ell_\varrho  \w{\phi} \in R^G_{\varrho}(L^2_\omega(\Omega)),$ consequently, $\ell_\varrho  \w{\phi} \in \mb P_\varrho(L^2_\omega(\Omega))$ by Lemma \ref{equal1}.
The following lemma ensures that $\mb P_\varrho(L^2_\omega(\Omega)) = \{\phi \in L_\omega^2(\Omega) : \phi = \ell_\varrho  \w{\phi}, \text{~where~} \w{\phi} \text{~is~} G\text{-invariant} \}.$



\begin{lem}\label{quo} Let $\varrho \in \w{G}_1$ and $f \in \mb P_\varrho(L^2_\omega(\Omega)).$ Then $f = \ell_\varrho~\widehat{f},$ where $\ell_\varrho$ is a generating polynomial of $R^G_{\varrho}(\mb C[z_1,\ldots,z_d])$ over the ring $\mb C[z_1,\ldots,z_d]^G$ and $\widehat{f}$ is $G$-invariant.
\end{lem}

\begin{proof} Since $\ell_\varrho$ vanishes only on a set of measure zero, we can write any $f \in \mb P_\varrho(L^2_\omega(\Omega)) $
as $$ f = \widehat{f} \ell_\varrho$$ where $\widehat{f} = \frac{f}{\ell_\varrho}.$ Clearly, $\widehat{f}$ is $G$-invariant.


\end{proof}
We recall an analytic version of Chevalley-Shephard-Todd theorem which allows us to state a few additional properties of the elements of the weighted Bergman space $\mb A^2_\omega(\Omega)$. First note that $\C[z_1,\ldots,z_d]$ is a free $\C[z_1,\ldots,z_d]^G$ module of rank $|G|$ \cite[Theorem 1, p. 110]{MR1890629}. Further, one can choose a  basis of $\C[z_1,\ldots,z_d]$
consisting of homogeneous polynomials. We choose the basis in the following manner. For each $\varrho\in \widehat G$,  $\mathbb{P}_{\varrho}(\C[z_1,\ldots,z_d])$ is a free module over $\C[z_1,\ldots,z_d]^G$ of rank $\deg (\varrho)^2$ \cite[Proposition II.5.3., p.28]{Pan}. Clearly, $\C[z_1,\ldots,z_d] = \oplus_{\varrho \in \widehat{G}} \mb P_\varrho (\C[z_1,\ldots,z_d]),$ where the direct sum is orthogonal direct sum borrowed from the Hilbert space structure of $\mb A^2_\omega(\Omega)$. Then we can choose a basis $\{\ell_{\varrho,i}: 1 \leq i \leq \deg(\varrho)^2\}$ of $\mathbb{P}_{\varrho}(\C[z_1,\ldots,z_d])$ over $\C[z_1,\ldots,z_d]^G$ for each $\varrho \in \widehat G$ such that together they form a basis $\{\ell_{\varrho,i}: \varrho \in \widehat{G} \text{~and~} 1 \leq i \leq \deg(\varrho)^2\}$ of $\C[z_1,\ldots,z_d]$ over $\C[z_1,\ldots,z_d]^G$. We work with such a choice of basis for the rest of discussion.
\begin{thm}[Analytic CST]\cite[p. 12, Theorem 3.12]{MR4404033}\label{acst}
Let $G$ be a finite group generated by pseudoreflections on $\mb C^d$ and $\Omega \subseteq \C^d$ be a $G$-space. For every holomorphic function $f$ on $\Omega$, there exist unique $G$-invariant holomorphic functions $\{f_{\varrho,i}: 1 \leq i \leq \deg(\varrho)^2\}_{\varrho \in \widehat{G}}$ such that $$f = \sum_{\varrho \in \widehat{G}} \sum_{i=1}^{\deg(\varrho)^2} f_{\varrho,i} \ell_{\varrho,i}.$$
\end{thm}
\begin{rem}\rm
With such a choice of basis, we have the following:
\begin{enumerate}
    \item For any element $f$ in $\mb A_\omega^2(\Omega),$ there exist unique $G$-invariant holomorphic functions $\{f_{\varrho,i}: 1 \leq i \leq \deg(\varrho)^2\}_{\varrho \in \widehat{G}}$ such that $f = \sum_{\varrho \in \widehat{G}} \sum_{i=1}^{\deg(\varrho)^2} f_{\varrho,i} \ell_{\varrho,i}$ and $\mb P_\varrho f = \sum_{i=1}^{\deg(\varrho)^2} f_{\varrho,i} \ell_{\varrho,i}$ for every $\varrho \in \widehat{G}.$
    \item For simplicity, we write $\ell_\varrho$ in the place of $\ell_{\varrho,1}$ if $\varrho$ is in $\widehat{G}_1.$ For every $\psi \in \mb P_\varrho(\mb A_\omega^2(\Omega)),$ there exists a unique $G$-invariant holomorphic function $\widehat{\psi}$ such that $\psi = \ell_\varrho~ \widehat{\psi}$ \cite[Lemma 3.1, Remark 3.3]{kag}. A description of $\ell_\varrho$ is given in  Lemma \ref{gencz}.
    \item  For $\varrho \in \widehat{G}_1,$ any other choice of basis of $\mathbb{P}_{\varrho}(\C[z_1,\ldots,z_d])$ as a free module over $\C[z_1,\ldots,z_d]^G$ is a constant multiple of $\ell_\varrho.$
\end{enumerate}
\end{rem}



\sf{Weighted $L^2$-spaces on $\bl\theta(\Omega)$:} 
Consider a $G$-invariant continuous weight function $\omega : \Omega \to (0,\infty).$ We write $\omega =\widetilde{\omega}\circ \bl \theta$ for a continuous map $\widetilde{\omega} : \bl \theta(\Omega) \to (0,\infty).$  For $\varrho \in\widehat{G}_1,$ we set $$\omega_\varrho(\bl \theta(\bl z)) = \frac{|\ell_\varrho(\bl z)|^2}{|J_{\bl \theta}(\bl z)|^2} ~\widetilde{\omega}(\bl \theta(\bl z)),$$
and  define the Hilbert space $L^2_{\omega_\varrho}(\bl \theta(\Omega))$ as following: $$L^2_{\omega_\varrho}(\bl \theta(\Omega)) = \{ \phi :\bl \theta(\Omega) \to \mb C ~| \int_{\bl \theta(\Omega)} |\phi(\bl u)|^2 \omega_\varrho(\bl u) dV(\bl u) < \infty \}.$$ 

\begin{lem}
For every $\varrho \in \widehat{G}_1$, the linear operator $\Gamma_\varrho: L^2_{\omega_\varrho}(\bl \theta(\Omega)) \to \mb P_\varrho(L_\omega^2(\Omega)),$ defined by,
\bea\label{gam} \Gamma_\varrho \phi = \frac{1}{\sqrt{|G|}}  (\phi \circ \bl \theta)\ell_\varrho,\,\,\,\, \phi \in L_{\omega_\varrho}^2(\bl \theta(\Omega)), \eea
is unitary.
\end{lem}
\begin{proof}
The polynomial $\ell_\varrho$ is in $\mb P_\varrho(\mb A^2_\omega(\Omega)) \subseteq\mb P_\varrho(L^2_\omega(\Omega))$, that is, $\sigma(\ell_\varrho)  = \chi_\varrho(\sigma) \ell_\varrho$ for all $\sigma \in G$ from Lemma \ref{equal1}. From the explicit expression of $\ell_\varrho$ in Lemma \ref{gencz}, we conclude that there exists a unique positive integer, say $m_\varrho,$ such that $\ell_\varrho^{m_\varrho}$ is $G$-invariant. Using CST theorem, $\ell_\varrho^{m_\varrho}(\bl z) = f_\varrho(\bl \theta(\bl z))$ for a unique polynomial $f_\varrho$. The function $\omega_\varrho$ can be written as \bea\label{weight1}\omega_\varrho(\bl u) = \frac{|f_\varrho(\bl u)|^\frac{2}{m_\varrho}}{|f_\sgn(\bl u)|^\frac{2}{m_{\rm sgn}}} \widetilde{\omega}(\bl u), \,\, \bl u \in \bl \theta(\Omega) \setminus N,\eea where $N=\{\bl \theta(\bl z) \in \mb C^d: J_{\bl \theta}(\bl z)=0\}.$ Then $\Gamma_\varrho$ is an isometry because for $\psi \in L^2_{\omega_\varrho}(\bl \theta(\Omega)),$
\bea\label{iso} \nonumber \norm{\Gamma_\varrho \psi}^2 &=& \frac{1}{|G|} \int_\Omega |\psi \circ \bl \theta(\bl z)|^2 |\ell_\varrho(\bl z)|^2 \omega(\bl z) dV(\bl z) \\ \nonumber &=&  \frac{1}{|G|} \int_{\Omega \setminus \bl \theta^{-1}(N)} |\psi \circ \bl \theta(\bl z)|^2 \frac{|f_\varrho(\bl \theta(\bl z))|^\frac{2}{m_\varrho}}{|f_{\rm sgn}(\bl \theta(\bl z))|^\frac{2}{m_{\rm sgn}}} |J_{\bl \theta}(\bl z)|^2 \widetilde{\omega}(\bl \theta(\bl z)) dV(\bl z) \\ &=& \int_{\bl \theta(\Omega) \setminus N} |\psi(\bl u)|^2 \omega_\varrho(\bl u) dV(\bl u) = \norm{\psi}^2, \eea the last equality follows since the set $N$ has Lebesgue measure $0.$

Let $\phi$ be in $\mb P_\varrho(L^2_\omega(\Omega)).$ Then there exists a $\widehat{\phi}$ such that $\phi = \ell_\varrho ~ \widehat{\phi} \circ \bl \theta$ (cf. Lemma \ref{quo}).  We are to show that $\widehat{\phi} \in L^2_{\omega_\varrho}(\bl \theta(\Omega)).$ From Equation \eqref{iso}, we get the norm of $\widehat{\phi}$ is equal to the norm of $\phi.$ Since the later is finite, so is the norm of $\widehat{\phi}$.
\end{proof}
Let $\m O(\Omega)$ be the set of all holomorphic functions on $\Omega.$ We denote the weighted Bergman space with weight function $\omega_\varrho$ by $$\mb A_{\omega_\varrho}^2(\bl \theta(\Omega)) = \{ \phi \in \m O(\bl \theta(\Omega)) : \int_{\bl \theta(\Omega)} |\phi(\bl u)|^2 \omega_\varrho(\bl u) dV(\bl u) < \infty \}.$$ The restriction operator $\Gamma_\varrho : \mb A_{\omega_\varrho}^2(\bl \theta(\Omega)) \to \mb P_\varrho(\mb A^2_\omega(\Omega))$ is surjective. Thus the weighted Bergman space $\mb A_{\omega_\varrho}^2(\bl \theta(\Omega))$ is isometrically isomorphic to $\mb P_\varrho (\mb A^2_\omega(\Omega))$ via the unitary operator $\Gamma_\varrho$ \cite{kag}. This enables us to make the following observations which are essential.
\begin{rem}\rm
\begin{enumerate}
\item For each $\varrho \in \widehat{G}_1,$ we have the following identity involving the reproducing kernel ${\mathfrak K}_{\omega_\varrho}$ of $\mb A_{\omega_\varrho}^2(\bl \theta(\Omega))$ and the reproducing kernel $\m B_\omega^{(\varrho)}$ of $\mb P_\varrho(\mb A^2_\omega(\Omega))$ \cite[p. 11, Equation 3.3]{kag}:
\Bea
\m B_\omega^{(\varrho)}(\bl z, \bl w) = \frac{1}{|G|}\ell_\varrho(\bl z) {\mathfrak K}_{\omega_\varrho}\big(\bl\theta(\bl z), \bl\theta(\bl w)\big)\ov{\ell_\varrho(\bl w)}.
\Eea
For a fixed $\bl w \in \Omega,$ we have \bea\label{ne}\nonumber\ov{\ell_\varrho(\bl w)}\Gamma_\varrho\big({\mathfrak K}_{\omega_\varrho}(\cdot, \bl\theta(\bl w))\big)(\bl z) &=& \frac{1}{\sqrt{|G|}}\ell_\varrho(\bl z) {\mathfrak K}_{\omega_\varrho}(\bl\theta(\bl z), \bl\theta(\bl w)\big)\ov{\ell_\varrho(\bl w)}\\&=& \sqrt{|G|}\m B_\omega^{(\varrho)}(\bl z, \bl w).\eea
\item Recall that the polynomial $\ell_\sgn=J_{\bl \theta}$ (cf. Corollary \ref{Jac}). Therefore, for $\widetilde{\omega} \equiv 1$ in Equation \eqref{weight1}, we get  $\omega_\sgn \equiv 1$ and thus $\mb A_{\omega_\sgn}^2(\bl \theta(\Omega))$ coincides with the Bergman space $\mb A^2(\bl \theta(\Omega)).$ Consequently, $\mb A^2(\bl \theta(\Omega))$ is isometrically isomorphic to $\mb P_\sgn (\mb A^2(\Omega))$ via the unitary operator $\Gamma_\sgn.$
\end{enumerate}
\end{rem}

\section{Toeplitz operators and the Weighted Bergman spaces}\label{mainsec}
In this section, we establish a number of identities involving Toeplitz operators on $\mb A_\omega^2(\Omega)$ and Toeplitz operators on $\mb A_{\omega_\varrho}^2(\bl \theta(\Omega))$ which enables us to study algebraic properties of Toeplitz operators on $\mb A_{\omega_\varrho}^2(\bl \theta(\Omega)).$ 
We start with two very useful lemmas. 
\begin{lem}\label{equi}
For every one-dimensional $\varrho \in \widehat{G}_1,$ the following diagram commutes:
\[ \begin{tikzcd}
L^2_{\omega_\varrho}(\bl \theta(\Omega)) \arrow{r}{P_{\omega_\varrho}} \arrow[swap]{d}{\Gamma_\varrho} & \mb A^2_{\omega_\varrho}(\bl \theta(\Omega)) \arrow{d}{\Gamma_\varrho} \\%
\mb P_\varrho(L_\omega^2(\Omega)) \arrow{r}{\widetilde{P}_{\omega,\varrho}}& \mb P_\varrho(\mb A_\omega^2(\Omega))
\end{tikzcd}
\]
where $P_{\omega_\varrho}$ and $\widetilde{P}_{\omega,\varrho}$ are the associated orthogonal projections and $\Gamma_\varrho$ is as defined in Equation \eqref{gam}.
\end{lem}
\begin{proof}
Note that for $f \in L^2_{\omega_\varrho}(\bl \theta(\Omega)),$ we have
\Bea (\Gamma_\varrho P_{\omega_\varrho} f)(\bl z) &=&  \frac{1}{\sqrt{|G|}}  (P_{\omega_\varrho} f \circ \bl \theta)(\bl z)\ell_\varrho(\bl z) \\
&=&\frac{1}{\sqrt{|G|}}\ell_\varrho(\bl z) \inner{f}{{\mathfrak K}_{\omega_\varrho}(\cdot, \bl\theta(\bl z))} \\&=& \frac{1}{\sqrt{|G|}}\ell_\varrho(\bl z) \inner{\Gamma_\varrho f}{\Gamma_\varrho\big({\mathfrak K}_{\omega_\varrho}(\cdot, \bl\theta(\bl z))\big)}\\ &=& \frac{1}{\sqrt{|G|}} \inner{\Gamma_\varrho f}{\sqrt{|G|}\m B_\omega^{(\varrho)}(\cdot,\bl z)} = (\widetilde{P}_{\omega,\varrho} \Gamma_\varrho f)(\bl z),\Eea
where the penultimate equality follows from Equation \eqref{ne}.
\end{proof}
If $f \in \mb P_\varrho(\mb A_\omega^2(\Omega)),$ then $\widetilde{u}f \in \mb P_\varrho(L_\omega^2(\Omega)).$ It follows from the orthogonal decomposition of $L_\omega^2(\Omega)$ in Equation \eqref{l2ortho} that $(T_{\widetilde{u}}f)(\bl z) = \inner{\widetilde{u}f}{\m B_\omega(\cdot, \bl z)} = \inner{\widetilde{u}f}{\m B_\omega^{(\varrho)}(\cdot, \bl z)}=\widetilde{P}_{\omega,\varrho}(\widetilde{u}f)(\bl z).$  Therefore, the subspace $\mb P_\varrho(\mb A_\omega^2(\Omega))$ remains invariant under $T_{\widetilde{u}}$ and the restriction operator $T_{\widetilde{u}}: \mb P_\varrho(\mb A_\omega^2(\Omega)) \to \mb P_\varrho(\mb A_\omega^2(\Omega))$ is given by $(T_{\widetilde{u}}f)= \widetilde{P}_{\omega,\varrho}(\widetilde{u}f).$ Moreover, the orthogonal complement of $\mb P_\varrho(\mb A_\omega^2(\Omega))$  is $\oplus_{\varrho' \in \widehat{G}, \varrho' \not\equiv \varrho} \mb P_{\varrho'} (\mb A^2_\omega(\Omega)),$ which is also invariant under $T_{\widetilde{u}}.$ Thus $\mb P_{\varrho} (\mb A^2_\omega(\Omega))$ is a reducing subspace for $T_{\widetilde{u}}.$ 

\begin{lem}\label{equi1}
Let $  \widetilde{u} \in L^\infty(\Omega)$ be a $G$-invariant function such that $\widetilde{u}= u\circ \bl \theta.$ For every $\varrho \in \widehat{G}_1,$ the following diagram commutes:
\[ \begin{tikzcd}
\mb A^2_{\omega_\varrho}(\bl \theta(\Omega)) \arrow{r}{T_u} \arrow[swap]{d}{\Gamma_\varrho} & \mb A^2_{\omega_\varrho}(\bl \theta(\Omega)) \arrow{d}{\Gamma_\varrho} \\%
\mb P_\varrho(\mb A_\omega^2(\Omega)) \arrow{r}{T_{\widetilde{u}}}& \mb P_\varrho(\mb A_\omega^2(\Omega))
\end{tikzcd}
\]
\end{lem}
\begin{proof}
Note that $\Gamma_\varrho( uf) = \frac{1}{\sqrt{|G|}}  (u \circ \bl \theta)(f \circ \bl \theta)~\ell_\varrho = \widetilde{u}~ \Gamma_\varrho(f).$ From Lemma \ref{equi}, we have
\Bea \Gamma_\varrho T_uf= \Gamma_\varrho(P_{\omega_\varrho}(uf)) = \widetilde{P}_{\omega,\varrho} (\Gamma_\varrho( uf)) =  \widetilde{P}_{\omega,\varrho} (\widetilde{u} \Gamma_\varrho( f)) = T_{\widetilde{u}}\Gamma_\varrho f.\Eea
\end{proof}

The following corollary is immediate from Lemma \ref{equi1}. 
\begin{cor}
Let $\widetilde{u}_1,\ldots,\widetilde{u}_n \in $ $L^\infty(\Omega)$ be $G$-invariant functions such that $\widetilde{u}_i = u_i \circ \bl \theta$ for $i=1,\ldots,n.$ For every $\varrho \in \widehat{G}_1,$ the product $T_{u_1}\cdots T_{u_n}$ on $\mb A^2_{\omega_\varrho}(\bl \theta(\Omega))$ is unitarily equivalent to $T_{\widetilde{u}_1}\cdots T_{\widetilde{u}_n}$ on $\mb P_\varrho(\mb A_\omega^2(\Omega))$.
\end{cor} 

If $\widetilde{u}$ is holomorphic, it is easy to see that the subspace $\ell_\varrho \cdot \mb P_{\tr}(\mb A^2_\omega(\Omega))$ is invariant under the restriction of the operator $T_{\widetilde{u}}$ on $\mb P_\varrho(\mb A^2_\omega(\Omega)).$
We prove that this continues to hold even when $\widetilde{u}$ is only a bounded function. 

Let $f \in \mb P_{\tr}(\mb A^2_\omega(\Omega)),$ then $ \ell_\varrho f \in \ell_\varrho \cdot \mb P_{\tr}(\mb A^2_\omega(\Omega)) \subseteq R^G_\varrho(\mb A^2_\omega(\Omega)) = \mb P_\varrho(\mb A^2_\omega(\Omega)),$ where the last equality follows from \cite[Lemma 3.1]{kag}. The density of $G$-invariant polynomials in $\mb P_{\tr}(\mb A^2_\omega(\Omega))$ implies that $\ell_\varrho \cdot \mb P_{\tr}(\mb A^2_\omega(\Omega))$ is dense in $\mb P_\varrho(\mb A^2_\omega(\Omega)).$

We consider $f \in \mb P_\varrho(\mb A^2_\omega(\Omega))$ such that $f = \ell_\varrho f_\varrho$ for $f_\varrho \in \mb P_{\tr}(\mb A^2_\omega(\Omega)).$  Then $\widetilde{u} f_\varrho \in \mb P_{\rm tr}(L_\omega^2(\Omega))$ using Lemma \ref{equal1} and the following holds: \Bea (T_{\widetilde{u}}f)(\bl z) = \inner{\widetilde{u} f}{\m B_\omega(\cdot, \bl z)} &=& \ell_\varrho(\bl z) \inner{\widetilde{u} f_\varrho}{\m B_\omega(\cdot, \bl z)}  \\ &=& \ell_\varrho(\bl z) \inner{\widetilde{u} f_\varrho}{\m B_\omega^{(\tr)}(\cdot, \bl z)}  \\ &=& \ell_\varrho(\bl z) \widetilde{P}_{\omega,\tr}(\widetilde{u} f_\varrho) (\bl z), \Eea 
where $\m B_\omega^{(\tr)}$ denotes the reproducing kernel of $\mb P_{\rm tr}(\mb A^2_\omega(\Omega)).$ Therefore, we have $T_{\widetilde{u}}(\ell_\varrho \cdot \mb P_{\tr}(\mb A^2_\omega(\Omega)) \subseteq \ell_\varrho \cdot \mb P_{\tr}(\mb A^2_\omega(\Omega))$ 
for every $\varrho \in \widehat{G}_1.$

This result can be extended to any representation $\varrho \in \widehat{G}$ with $\deg( \varrho) > 1.$ 
We consider a basis $\{\ell_{\varrho,i}\}_{i=1}^{\deg( \varrho)^2}$ of $\mb P_\varrho (\mb C[z_1,\ldots,z_d])$ as a free module over $\mb C[z_1,\ldots,z_d]^G.$ Since $\sum_{i=1}^{\deg( \varrho)^2} \ell_{\varrho,i} \cdot \mb C[z_1,\ldots,z_d]^G $ is dense in $\sum_{i=1}^{\deg( \varrho)^2} \ell_{\varrho,i} \cdot \mb P_{\rm tr} (\mb A_\omega^2(\Omega))$ and $\sum_{i=1}^{\deg( \varrho)^2} \ell_{\varrho,i} \cdot \mb P_{\rm tr} (\mb A^2_\omega(\Omega))$ is contained in $\mb P_\varrho (\mb A^2_\omega(\Omega)),$ we get that $\sum_{i=1}^{\deg( \varrho)^2} \ell_{\varrho,i} \cdot \mb P_{\rm tr} (\mb A^2_\omega(\Omega))$ is dense in $\mb P_\varrho (\mb A^2_\omega(\Omega)).$ For $f = \sum_{i=1}^{\deg( \varrho)^2} \ell_{\varrho,i} f_{\varrho,i},$ such that  $f_{\varrho,i} \in \mb P_{\rm tr} (\mb A^2_\omega(\Omega)),$ we conclude the following:
\bea\label{sum}
\nonumber(T_{\widetilde{u}}f)(\bl z) = \inner{\widetilde{u} f}{\m B_\omega(\cdot, \bl z)} &=& \inner{\sum_{i=1}^{\deg( \varrho)^2} \ell_{\varrho,i} \widetilde{u} f_{\varrho,i}}{\m B_\omega(\cdot, \bl z)} \\&=&\nonumber \sum_{i=1}^{\deg( \varrho)^2} \inner{\widetilde{u} f_{\varrho,i}}{M^*_{\ell_{\varrho,i}}\m B_\omega(\cdot, \bl z)} \\&=& \nonumber \sum_{i=1}^{\deg( \varrho)^2} \ell_{\varrho,i}(\bl z) \inner{\widetilde{u} f_{\varrho,i}}{\m B_\omega(\cdot, \bl z)} \\ &=&\nonumber \sum_{i=1}^{\deg( \varrho)^2} \ell_{\varrho,i}(\bl z) \inner{\widetilde{u} f_{\varrho,i}}{\m B_\omega^{(\tr)}(\cdot, \bl z)} \\ &=& \sum_{i=1}^{\deg( \varrho)^2} \ell_{\varrho,i}(\bl z) \widetilde{P}_{\omega,\tr}(\widetilde{u} f_{\varrho,i})(\bl z).
\eea
Hence, each $\sum_{i=1}^{\deg( \varrho)^2} \ell_{\varrho,i} \cdot \mb P_{\rm tr} (\mb A_\omega^2(\Omega))$ remains invariant for the operator $T_{\widetilde{u}}.$

\subsection{Generalized zero-product problem}
Let $\widetilde{u}, \widetilde{v}$ and $\widetilde{q}$ be $G$-invariant functions in $L^\infty(\Omega).$ It is evident that if $T_{\widetilde{u}}T_{\widetilde{v}}=T_{\widetilde{q}}$ on $\mb A^2_\omega(\Omega),$ then the restriction operator $T_{\widetilde{u}}T_{\widetilde{v}}=T_{\widetilde{q}}$ on $\mb P_\varrho(\mb A^2_\omega(\Omega))$ for every $\varrho \in \widehat{G}$. Interestingly enough, the converse holds as well with a weaker hypothesis.
\begin{lem}\label{prep}
If $T_{\widetilde{u}}T_{\widetilde{v}}=T_{\widetilde{q}}$ on $\mb P_\varrho(\mb A^2_\omega(\Omega))$ for at least an irreducible representation $\varrho \in \widehat{G},$ then $T_{\widetilde{u}}T_{\widetilde{v}}=T_{\widetilde{q}}$ on $\mb A^2_\omega(\Omega).$
\end{lem}
\begin{proof}
 Suppose that the assumption holds for the representation $\varrho \in \widehat{G}.$ 
 Consider an element $f = \sum_{i=1}^{\deg( \varrho)^2} \ell_{\varrho,i} \widehat{f}_{\varrho,i}$ such that $\widehat{f}_{\varrho,i} \in \mb P_{\tr}(\mb A^2_\omega(\Omega)).$ Then from Equation \eqref{sum}, it is clear that $ T_{\widetilde{u}}T_{\widetilde{v}}f=T_{\widetilde{q}}f$ implies \Bea \sum_{i=1}^{\deg( \varrho)^2} \ell_{\varrho,i} \widetilde{P}_{\omega,\tr}(\widetilde{u} ~ \widetilde{P}_{\omega,\tr}(\widetilde{v}\widehat{f}_{\varrho,i})) =\sum_{i=1}^{\deg( \varrho)^2}\ell_{\varrho,i} \widetilde{P}_{\omega,\tr}(\widetilde{q} \widehat{f}_{\varrho,i}).\Eea In case, $\varrho \in \widehat{G}_1,$ we get $$\ell_{\varrho,1} \widetilde{P}_{\omega,\tr}(\widetilde{u} ~ \widetilde{P}_{\omega,\tr}(\widetilde{v}\widehat{f}_{\varrho,1}) =\ell_{\varrho,1} \widetilde{P}_{\omega,\tr}(\widetilde{q} \widehat{f}_{\varrho,1})$$ for $\widehat{f}_{\varrho,1} \in \mb P_{\tr}(\mb A^2_\omega(\Omega)),$ equivalently, $\widetilde{P}_{\omega,\tr}(\widetilde{u}~ \widetilde{P}_{\omega,\tr}(\widetilde{v} \widehat{f}))= \widetilde{P}_{\omega,\tr}(\widetilde{q} \widehat{f})$ for every $\widehat{f} \in \mb P_{\tr}(\mb A^2_\omega(\Omega)).$ If $\deg(\varrho)>1,$ we take $\widehat{f}_{\varrho,i} = 0$ for $i=2,\ldots,\deg( \varrho)^2$ and then repeat the above argument to observe the same. 
 
Let $f = \sum_{\varrho \in \widehat{G}} \sum_{i=1}^{\deg( \varrho)^2} \ell_{\varrho,i} f_{\varrho,i},$ for some $f_{\varrho,i} \in \mb P_{\tr}(\mb A^2_\omega(\Omega)).$ Then we get $$T_{\widetilde{u}}T_{\widetilde{v}}f =\sum_{\varrho \in \widehat{G}} \sum_{i=1}^{\deg( \varrho)^2} \ell_{\varrho,i} \widetilde{P}_{\omega,\tr}(\widetilde{u}~ \widetilde{P}_{\omega,\tr}(\widetilde{v} f_{\varrho,i}))=\sum_{\varrho \in \widehat{G}} \sum_{i=1}^{\deg( \varrho)^2} \ell_{\varrho,i} \widetilde{P}_{\omega,\tr}(\widetilde{q} f_{\varrho,i}) = T_{\widetilde{q}}f$$ on a dense subset of $\mb A^2_\omega(\Omega).$ Hence, the result follows.
\end{proof}
We note down an useful observation from the proof of Lemma \ref{prep} in the following corollary. 
 

\begin{cor}\label{all}
Let $\widetilde{u}, \widetilde{v}$ and $\widetilde{q}$ be $G$-invariant functions in $L^\infty(\Omega).$ If $T_{\widetilde{u}}T_{\widetilde{v}}=T_{\widetilde{q}}$ on $\mb P_\mu(\mb A^2_\omega(\Omega))$ for at least one representation $\mu \in \widehat{G},$ then $T_{\widetilde{u}}T_{\widetilde{v}}=T_{\widetilde{q}}$ on $\mb P_\varrho(\mb A^2_\omega(\Omega))$ for every $\varrho \in \widehat{G}.$
\end{cor}
\begin{proof}
The result follows since the assumption boils down to $\widetilde{P}_{\omega,\tr}(\widetilde{u} ~ \widetilde{P}_{\omega,\tr}(\widetilde{v}\widehat{f}) = \widetilde{P}_{\omega,\tr}(\widetilde{q} \widehat{f})$ for every $\widehat{f} \in \mb P_{\tr}(\mb A^2_\omega(\Omega))$ and consequently $$\sum_{i=1}^{\deg( \varrho)^2} \ell_{\varrho,i} \widetilde{P}_{\omega,\tr}(\widetilde{u} ~ \widetilde{P}_{\omega,\tr}(\widetilde{v}\widehat{f}_{\varrho,i})) =\sum_{i=1}^{\deg( \varrho)^2}\ell_{\varrho,i} \widetilde{P}_{\omega,\tr}(\widetilde{q} \widehat{f}_{\varrho,i})$$ for every $f = \sum_{i=1}^{\deg( \varrho)^2} \ell_{\varrho,i} \widehat{f}_{\varrho,i}$ where $\widehat{f}_{\varrho,i} \in \mb P_{\rm tr}(\mb A^2_\omega(\Omega)).$
\end{proof}


\begin{proof}[\textbf {Proof of Theorem  \ref{main1}}]
Since $\Gamma_\mu:\mb A^2_{\omega_\mu}(\bl \theta(\Omega))\to\mb P_\mu(\mb A^2_\omega(\Omega))$ is a unitary operator, for each $\widetilde{f} \in \mb P_\mu(\mb A^2_\omega(\Omega))$ there exists $f \in \mb A^2_{\omega_\mu}(\bl \theta(\Omega))$ such that $\Gamma_\mu(f)= \widetilde{f}.$ We note that for every $\widetilde{f} \in \mb P_\mu(\mb A^2_\omega(\Omega)),$  \bea \label{hh1} T_{\widetilde{u}}T_{\widetilde{v}} \widetilde{f}= T_{\widetilde{u}}T_{\widetilde{v}} \Gamma_\mu(f) &=& T_{\widetilde{u}} \Gamma_\mu (T_v f)\\ &=& \nonumber \Gamma_\mu (T_uT_vf)\\ &=& \label{hh}  \Gamma_\mu(T_q f) \\&=& \nonumber T_{\widetilde{q}} \Gamma_\mu(f) = T_{\widetilde{q}} \widetilde{f},\eea 
where the equality in Equation \eqref{hh1} and the penultimate equality follow from Lemma \ref{equi1}, the equality in Equation \eqref{hh} follows from the assumption. This with Lemma \ref{prep} immediately proves 1(ii). 

Furthermore, from Corollary \ref{all} we get $T_{\widetilde{u}}T_{\widetilde{v}}=T_{\widetilde{q}}$ on $\mb P_\varrho(\mb A^2_\omega(\Omega))$ for every $\varrho \in \widehat{G}.$ A repetitive use of Lemma \ref{equi1} in a similar argument as above leads us to 1(i).

Conversely, if the assumption holds then $T_{\widetilde{u}}T_{\widetilde{v}}=T_{\widetilde{q}}$ on $\mb P_\varrho(\mb A^2_\omega(\Omega))$ for every $\varrho \in \widehat{G}.$ 
Then a repetitive application of Lemma \ref{equi1} as above proves part 2. 
\end{proof}
An important observation is noted down in the following corollary for the particular case of sign representation of $G$ and weight function $\omega \equiv 1.$
\begin{cor}\label{cor1}
Let $\Omega$ be a $G$-space under the action of a finite pseudoreflection group $G$ and $\bl \theta: \Omega \to \bl \theta(\Omega)$ be a basic polynomial map associated to $G.$ For $\widetilde{u} = u \circ \bl \theta$, $\widetilde{v} = v \circ \bl \theta$ and $\widetilde{q} = q \circ \bl \theta \in L^\infty(\Omega),$ $T_{\widetilde{u}}T_{\widetilde{v}}=T_{\widetilde{q}}$ on $\mb A^2(\Omega)$ if and only if $T_uT_v=T_q$ on $\mb A^2(\bl \theta(\Omega)).$
\end{cor}

\begin{rem}
For $\widetilde{q} \equiv 0,$ Corollary \ref{cor1} connects the zero-product problem of two Toeplitz operators on $\mb A^2(\Omega)$ to that of Toeplitz operators on $\mb A^2(\bl \theta(\Omega)).$
In fact, minor modification of the above method shows that $T_{\widetilde{u_1}}\cdots T_{\widetilde{u_N}}=0$ on $\mb A^2(\Omega)$ if and only if $T_{u_1}\cdots T_{u_N}=0$ on $\mb A^2(\bl \theta(\Omega))$ for $\widetilde{u_i} = u_i \circ \bl \theta \in L^\infty(\Omega),$ $i=1,\ldots,N.$ 
\end{rem}
\subsection{Characterization of a commuting pair of Toeplitz operators} Theorem \ref{main2} characterizes commuting pairs of Toeplitz operators on the weighted Bergman spaces of $\bl \theta(\Omega)$ in terms of commuting pairs of Toeplitz operators on $\mb A^2_\omega(\Omega).$ We give a proof of Theorem \ref{main2} here. Let $\widetilde{u}$ and $\widetilde{v}$ be $G$-invariant functions in $L^\infty(\Omega).$ If  $T_{\widetilde{u}}T_{\widetilde{v}} = T_{\widetilde{v}}T_{\widetilde{u}}$ on $\mb A^2_\omega(\Omega),$ then $T_{\widetilde{u}}T_{\widetilde{v}} = T_{\widetilde{v}}T_{\widetilde{u}}$ on each $\mb P_\varrho(\mb A^2_\omega(\Omega)).$ 
Also the converse holds with a rather weaker assumption.

\begin{lem}\label{com}
If $T_{\widetilde{u}}T_{\widetilde{v}}=T_{\widetilde{v}}T_{\widetilde{u}}$ on $\mb P_\varrho(\mb A^2_\omega(\Omega))$ for at least one irreducible representation $\varrho \in \widehat{G},$ then $T_{\widetilde{u}}T_{\widetilde{v}}=T_{\widetilde{v}}T_{\widetilde{u}}$ on $\mb A^2_\omega(\Omega).$
\end{lem}
Since the proof is similar to that of Lemma \ref{prep}, we omit the details.
Clearly, if $T_{\widetilde{u}}T_{\widetilde{v}}=T_{\widetilde{v}}T_{\widetilde{u}}$ on $\mb P_\mu(\mb A^2_\omega(\Omega))$ for at least one irreducible representation $\mu \in \widehat{G},$ then $T_{\widetilde{u}}T_{\widetilde{v}}=T_{\widetilde{v}}T_{\widetilde{u}}$ on $\mb P_\varrho(\mb A^2_\omega(\Omega))$ for every $\varrho \in \widehat{G}.$ 
\begin{proof}[\textbf {Proof of Theorem  \ref{main2}}]
From the assumption and Lemma \ref{equi1}, we have $T_{\widetilde{u}}T_{\widetilde{v}}=T_{\widetilde{v}}T_{\widetilde{u}}$ on $\mb P_\mu(\mb A^2_\omega(\Omega))$ and then Lemma \ref{com} proves 1(ii). 

Since 1(ii) implies that $T_{\widetilde{u}}T_{\widetilde{v}}=T_{\widetilde{v}}T_{\widetilde{u}}$ on $\mb P_\varrho(\mb A^2_\omega(\Omega))$ for $\varrho \in \widehat{G},$ we apply Lemma \ref{equi1} to conclude 1(i).

Converse is straightforward with a similar application of Lemma \ref{equi1} as above.
\end{proof}
\begin{cor}
Let $\widetilde{u} = u \circ \bl \theta$ and $\widetilde{v} = v \circ \bl \theta$ be $G$-invariant functions in $L^\infty(\Omega).$ Then $T_uT_v=T_vT_u$ on $\mb A^2(\bl \theta(\Omega))$ if and only if $T_{\widetilde{u}}T_{\widetilde{v}}=T_{\widetilde{v}}T_{\widetilde{u}}$ on $\mb A^2(\Omega).$
\end{cor}

\begin{rem}[Hankel Operators]For $\widetilde{u} \in L^\infty(\Omega),$ the Hankel operator $H_{\widetilde{u}} : \mb A^2_\omega(\Omega) \to \mb A^2_\omega(\Omega)^\perp$ is defined by $H_{\widetilde{u}}f = \widetilde{u}f- \widetilde{P}_{\omega}(\widetilde{u}f),$ where $\widetilde{P}_\omega : L^2_\omega(\Omega) \to \mb A^2_\omega(\Omega)$ is the orthogonal projection. We note that the unitary $\Gamma_\varrho: \mb A^2_{\omega_\varrho}(\bl \theta(\Omega)) \to \mb P_\varrho(\mb A^2_\omega(\Omega))$ intertwines $H_{\widetilde{u}}$ and $H_u,$ where $\widetilde{u}=u \circ \bl \theta$ and $\varrho \in \widehat{G}_1.$ Then a similar method to the above helps us to study algebraic properties of the Hankel operators on the Bergman spaces of the quotient domains.
\end{rem} 

\section{Application to certain quotient domains}\label{appl}
 Let $\mb D^d=\{\bl z \in \mb C^d: |z_1|, \ldots , |z_d| < 1\}$ be the polydisc in $\mb C^d$ and $\mb B_d$ be the open unit ball with respect to the $\ell^2$-norm induced by the standard inner product on $\mb C^d.$ 
 In this section, we apply the general results (Theorem \ref{main1} and Theorem \ref{main2}) on the zero-product problem and on the characterization of commuting of Toeplitz operators to the weighted Bergman spaces on specific quotient domains $\Omega/G$ where $\Omega=\mb D^d$ or $\mb B_d.$ 
 
First we recall the zero-product theorems for Toeplitz operators on the Bergman spaces $\mb A^2(\mb D^d)$ and $\mb A^2(\mb B_d)$ from \cite{MR2294274} and \cite{MR2214579}, respectively. 
A function $\phi$ is called a $d$-harmonic on a domain $\Omega \subset \mathbb C^d$ if $$\frac{\del^2\phi}{\del z_i \del\overbar{z}_i}=0~\text{for all}~ i=1,\ldots,d.$$ Let $h^\infty(\Omega)$ denote the class of all bounded $d$-harmonic functions on $\Omega.$

A function $\phi$ is called pluriharmonic on $\Omega$ if $$\frac{\del^2\phi}{\del z_i \del\overbar{z}_j}=0~\text{for  all}~i,j=1,\ldots,d.$$

\begin{thm}\cite[p. 45, Theorem 1.1.]{MR2294274}\label{d}
Suppose that $u, v \in h^\infty(\mb D^d)$ are continuous on $\mb D^d \cup W$ for some relatively open subset $W \text{~of~} \mb T^d.$ If $T_uT_v=0$ on $\mb A^2(\mb D^d)$ then either $u=0$ or $v=0.$
\end{thm}
\begin{thm}\cite[p. 309, Theorem 1.1.]{MR2214579}\label{b}
Suppose that $u, v \in h^\infty(\mb B_d)$ are continuous on $\mb B_d \cup W$ for some relatively open subset $W \text{~of~} \del_S\mb B_d.$ If $T_uT_v=0$ on $\mb A^2(\mb B_d)$ then either $u=0$ or $v=0.$
\end{thm}
\begin{defn}\label{ph1}
Let $\Omega$ be a $G$-invariant domain and $\bl \theta: \Omega \to \bl \theta(\Omega)$ be a basic polynomial map associated to the finite pseudoreflection group $G.$ A function $\phi$ defined on $\bl \theta(\Omega)$ is said to be $G$-pluriharmonic on $\bl \theta(\Omega)$ if 
$\phi \circ \bl \theta$ is a pluriharmonic function on $\Omega.$
\end{defn}
Suppose $\widetilde{\phi}$ be a pluriharmonic function on $\Omega.$ Then we write $\phi \circ \bl \theta  = \sum_{\sigma \in G} \widetilde{\phi} \circ \sigma$ and $\phi$ is a $G$-pluriharmonic function on $\bl \theta(\Omega).$ The set of all bounded $G$-pluriharmonic functions on $\bl \theta(\Omega)$ is denoted by $h_G^\infty(\bl \theta(\Omega)).$  

Here we recall some requisite facts on the Shilov boundary $\del_S\Omega$ of a bounded domain $\Omega$. The Shilov boundary of the polydisc $\mb D^d$ is the $d$-torus $\mb T^d.$ The Shilov boundary of the unit ball $\mb B_d$ coincides with its topological boundary. We note that $\bl \theta : \Omega \to \bl \theta(\Omega)$ is a proper holomorphic map which can be extended to a proper holomorphic map of the same multiplicity from $\Omega'$ to $\bl \theta(\Omega)',$ where the open sets $\Omega'$ and $\bl \theta(\Omega)'$ contain $\ov{\Omega}$ and $\ov{\bl \theta(\Omega)},$ respectively. Then \cite[p. 100, Corollary 3.2]{MR3001625} states that $\bl \theta^{-1}(\del_S\bl \theta(\Omega)) = \del_S\Omega.$ So for a relatively open subset $W$ of $\del_S\bl \theta(\Omega),$ there exists at least one relatively open subset $W'$ of the Shilov boundary of $\Omega$ such that $\bl \theta(W') = W.$ Thus if a function $u$ is continuous on $\bl \theta(\Omega)\cup W$ then so is $\widetilde{u}=u \circ \bl \theta$ on $\Omega \cup W'.$ 


\begin{proof}[\textbf  {Proof of Theorem \ref{result}}]
From the hypothesis and Theorem \ref{main1}, we note that  $T_{\widetilde{u}}T_{\widetilde{v}} = 0$ on $\mb A^2(\Omega)$. The assumption on $u$ and $v$ implies that $\widetilde{u} = u \circ \bl \theta$ and $\widetilde{v} = v \circ \bl \theta \in h^\infty(\Omega).$ The functions $u$ and $v$ are continuous on $\Omega \cup W^\i$ for some relatively open subset $W^\i$ of $\del_S\Omega.$ For $\Omega = \mb D^d,$ we use Theorem \ref{d} and for $\Omega = \mb B_d,$ we use Theorem \ref{b} to conclude that either ${\widetilde{u}}=0$ or ${\widetilde{v}}=0.$ This completes the proof.
\end{proof}
It is worth mentioning the particular case of the sign representation.
\begin{cor}
Let $\Omega = \mb D^d$ or $\mb B_d.$ Suppose that $u, v \in h_G^\infty(\bl \theta(\Omega))$ are continuous on $\bl \theta(\Omega) \cup W$ for some relatively open subset $W \text{~of~} \del_S\bl \theta(\Omega).$ If $T_uT_v=0$ on $\mb A^2(\bl \theta(\Omega))$ then either $u=0$ or $v=0.$
\end{cor}

\subsection{Symmetrized polydisc}
The permutation group on $d$ symbols is denoted by $\mathfrak S_d$. The group $\mathfrak S_d$ acts on $\mb C^d$ by permuting its coordinates, that is, 
\Bea
\sigma \cdot (z_1,\ldots,z_d) = (z_{\sigma^{-1}(1)},\ldots,z_{\sigma^{-1}(d)}) \,\, \text{for}\,\,  \sigma \in \mathfrak S_d\,\, \text{and}\,\, (z_1,\ldots,z_d) \in \mb C^d.
\Eea
Clearly, the open unit polydisc $\mb D^d$ is invariant under the action of the group $\mathfrak S_d.$ Let $s_k$ denote the elementary symmetric polynomials of degree $k$ in $d$ variables, for $k=1,\ldots, d.$ 
The symmetrization map \bea\label{sym}\bl s :=(s_1,\ldots,s_d) : \mb C^d \to \mb C^d\eea is a basic polynomial map associated to the pseudoreflection group $\mathfrak S_d.$ The domain $\mb G_d:=\bl s(\mb D^d)$ is known as the symmetrized polydisc. 

The symmetric group $\mathfrak{S}_d$ has only two one-dimensional representation in $\widehat{\mathfrak{S}}_d.$ Those are the sign representation and the trivial representation of $\mathfrak{S}_d$. We take $\omega \equiv 1$ in Equation \eqref{wei} and thus we have $\omega_\sgn \equiv 1$ and $\omega_{\tr}(\bl s(\bl z)) = \frac{1}{\prod_{i<j} |z_i-z_j|^2}, \,\,\bl z \in \mb D^d.$ 
\subsubsection{Zero-product problem} We have the following zero-product theorem for Toeplitz operators on $\mb A^2_{\omega_{\tr}}(\mb G_d)$ and $\mb A^2(\mb G_d)$ as a consequence of Theorem \ref{result}.
\begin{cor}
Suppose that $u, v \in L^\infty(\mb G_d)$ are continuous on $\mb G_d \cup W$ for some relatively open subset $W \text{~of~} \bl s(\mb T^d).$ Assume that $\widetilde{u} = u \circ \bl s, \widetilde{v} = v \circ \bl s$ belong to $h^\infty(\mathbb D^d).$
\begin{enumerate}
    \item If $T_uT_v=0$ on $\mb A^2(\mb G_d),$ then either $u=0$ or $v=0.$
    \item If $T_uT_v=0$ on $\mb A^2_{\omega_{\tr}}(\mb G_d),$ then either $u=0$ or $v=0.$
\end{enumerate}
\end{cor}

\subsubsection{Commuting pairs of Toeplitz operators} A bounded $\mathfrak S_d$-pluriharmonic function $u$ on $\mb G_d$ can be written as $ \widetilde{u} = u \circ \bl s$ for a $\mathfrak S_d$-invariant pluriharmonic function $\widetilde{u}$ on $\mb D^d.$  It is known that a bounded pluriharmonic function $\widetilde{u}$ on $\mb D^d$ can be written as $\widetilde{u}=\widetilde{f}+\ov{\widetilde{g}}$ for $\widetilde{f},\widetilde{g} \in \m O(\mb D^d),$ where $\m O(\Omega)$ denotes the set of all holomorphic functions on $\Omega.$ Moreover, if $\widetilde{u}$ is $\mathfrak S_d$-invariant, then both $\widetilde{f},\widetilde{g}$ are $\mathfrak S_d$-invariant and from \cite[Subsection 3.1.1.]{MR4404033}, we have $\widetilde{f} = f \circ \bl s$ and $\widetilde{g} = g \circ \bl s$ for $f,g \in \m O(\mb G_d).$ Thus $u \circ \bl s(\bl z)=\widetilde{u}(\bl z) =  \widetilde{f}(\bl z)+\ov{\widetilde{g}}(\bl z) =  f\circ \bl s(\bl z)+\ov{g}\circ \bl s(\bl z)$ for $f, g \in \m O(\mb G_d).$
Moreover, if $u$ is $\mathfrak{S}_d$-pluriharmonic on $\mb G_d,$ then $u\circ \bl s$ is a symmetric $d$-harmonic function on $\mb D^d.$ However, it is not clear whether $u$ is an $\mathfrak{S}_d$-pluriharmonic function on $\mb G_d$, whenever $u\circ \bl s$ is a symmetric $d$-harmonic function on $\mb D^d$. Let $\m X(\mb G_d)=\{u : u \text{ is }\mathfrak{S}_d\text{-pluriharmonic on } \mb G_d \text{ and } u\circ \bl s \text{ is } d\text{-harmonic on } \mb D^d\}.$
\begin{prop}
Let $u,v \in L^\infty(\mb G_d)$ be $\mathfrak{S}_d$-pluriharmonic functions such that $u = f+\ov{g}$ and $v=h+\ov{k}$ for $f,g,h,k \in \m O(\mb G_d).$ Then if $f\ov{k} - h\ov{g}$ is $\mathfrak{S}_d$-pluriharmonic on $\mb G_d$ then \begin{enumerate}
    \item $T_uT_v = T_vT_u$ on $\mb A^2(\mb G_d)$ and
    \item $T_uT_v = T_vT_u$ on $\mb A^2_{\omega_{\tr}}(\mb G_d).$
\end{enumerate} Moreover, $T_uT_v = T_vT_u$ on $\mb A^2(\mb G_d)$ (or on $\mb A^2_{\omega_{\tr}}(\mb G_d)$) if and only if $f\ov{k} - h\ov{g} \in \m X(\mb G_d).$
\end{prop}\begin{proof}

Let $u \circ \bl s(\bl z)=\widetilde{u}(\bl z) =  \widetilde{f}(\bl z)+\ov{\widetilde{g}}(\bl z)$ and $v \circ \bl s(\bl z)=\widetilde{v}(\bl z) =  \widetilde{h}(\bl z)+\ov{\widetilde{k}}(\bl z).$ From the assumption, it follows that $\widetilde{u}$ and $\widetilde{v} \in L^\infty(\mb D^d)$ and those are pluriharmonic on $\mb D^d.$ Note that $f\ov{k} - h\ov{g}$ is $\mathfrak{S}_d$-pluriharmonic on $\mb G_d$ implies that $\widetilde{f}\ov{\widetilde{k}} - \widetilde{h}\ov{\widetilde{g}}$ is $d$-harmonic on $\mb D^d$. Therefore, $T_{\widetilde{u}}T_{\widetilde{v}} = T_{\widetilde{v}}T_{\widetilde{u}}$ on $\mb A^2(\mb D^d)$ from \cite[p. 1728, Theorem 1.1]{MR2031039} and then using Theorem \ref{main2}, we get \emph{1.} and \emph{2.}

Conversely, if $T_uT_v = T_vT_u$ on $\mb A^2(\mb G_d)$ (or on $\mb A^2_{\omega_{\tr}}(\mb G_d)$), then combining Theorem \ref{main2} and \cite[p. 1728, Theorem 1.1]{MR2031039}, one gets $\widetilde{f}\ov{\widetilde{k}} - \widetilde{h}\ov{\widetilde{g}}$ is $d$-harmonic on $\mb D^d$. Thus the result follows.
\end{proof}


\subsection{Rudin's domains}
A family of quotient domains of the form $\mb B_d / G$ is described in \cite{MR667790}, where the group $G$ is a conjugate to a finite pseudoreflection group. Following \cite[p. 427]{MR742433}, we refer to such domains as Rudin's domains. The domain $\Omega \subset \mb C^d$ is a Rudin's domain if and only if there exists a proper holomorphic map $\bl F: \mb B_d \to \Omega.$
\subsubsection{Zero-product problem} We have the following zero-product theorem for Toeplitz operators on the weighted Bergman space on Rudin's domains using Theorem \ref{result}.
\begin{cor}\label{rud}
Let $\Omega$ be a Rudin's domain which is biholomorphic to $\mb B_d/G$ for some finite pseudoreflection group $G.$ Suppose that $u, v \in h_G^\infty(\Omega)$ are continuous on $\Omega \cup W$ for some relatively open subset $W \text{~of~} \del_S\Omega.$ If $T_uT_v=0$ on $\mb A_{\omega_\varrho}^2(\Omega)$ for a one-dimensional representation $\varrho$ of $G,$ then either $u=0$ or $v=0.$
\end{cor}
In particular for the sign representation, if $T_uT_v=0$ on $\mb A^2(\Omega)$ then either $u=0$ or $v=0,$ provided $u$ and $v$ follow the assumption in Corollary \ref{rud}.

\subsubsection{Commuting pairs of Toeplitz operators} 
The following result describes necessary and sufficient conditions for a pair of commuting Toeplitz operators on the Bergman spaces of Rudin's domains. We reduce the problem to the characterization of commuting Toeplitz operators on the Bergman space on the unit ball and then a direct application of \cite[p. 1597, Theorem 2.1]{MR1443898} and \cite[Theorem 3.1]{MR4404033} proves the result. 
\begin{prop}
Let $\Omega$ be a Rudin's domain which is biholomorphic to $\mb B_d/G$ for some finite pseudoreflection group $G.$ Suppose that $u, v \in h_G^\infty(\Omega),$ then $$T_uT_v = T_v T_u$$ on $\mb A_{\omega_\varrho}^2(\Omega)$ for a one-dimensional representation $\varrho$ of $G,$ if and only if one of the following holds:
\begin{enumerate}
    \item Both $u$ and $v$ are holomorphic on $\Omega.$
    \item Both $\ov{u}$ and $\ov{v}$ are holomorphic on $\Omega.$
    \item Either $u$ or $v$ is constant on $\Omega.$
    \item For a nonzero constant $b,$ $u-bv$ is constant on $\Omega.$ 
\end{enumerate}
\end{prop}
\begin{proof}
Note that for $u$ and $v \in h_G^\infty(\Omega),$ there exist bounded pluriharmonic functions $\widetilde{u}=u \circ \bl F$ and $\widetilde{v} = v \circ \bl F$ on $\mb B_d$, where $\bl F:\mb B_d \to \Omega$ is a proper holomorphic map. Clearly, if $u$ (or $\ov{u}$) is holomorphic on $\Omega$, then $\widetilde{u}$ (or $\ov{\widetilde{u}}$) is holomorphic on $\mb B_d.$ On the other hand, if $\widetilde{u}$ (or $\ov{\widetilde{u}}$) is holomorphic on $\mb B_d,$ then so is $u$ (or $\ov{u}$) on $\Omega$ from \cite[Theorem 3.1]{MR4404033}.

Also, $T_uT_v = T_v T_u$ on $\mb A_{\omega_\varrho}^2(\Omega)$ for a one-dimensional representation $\varrho$ of $G$ implies $T_{\widetilde{u}}T_{\widetilde{v}}=T_{\widetilde{v}}T_{\widetilde{u}}$ on $\mb A^2(\mb B_d)$ (cf. Theorem \ref{main2}) and vice-versa. From \cite[p. 1597, Theorem 2.1]{MR1443898}, it is known that  $T_{\widetilde{u}}T_{\widetilde{v}}=T_{\widetilde{v}}T_{\widetilde{u}}$ on $\mb A^2(\mb B_d)$ if and only if one of the following holds:
\begin{enumerate}
    \item Both $\widetilde{u}$ and $\widetilde{v}$ are holomorphic on $\mb B_d.$
    \item Both $\ov{\widetilde{u}}$ and $\ov{\widetilde{v}}$ are holomorphic on $\mb B_d.$
    \item Either $\widetilde{u}$ or $\widetilde{v}$ is constant on $\mb B_d.$
    \item For a nonzero constant $b,$ $\widetilde{u}-b\widetilde{v}$ is constant on $\mb B_d.$ 
\end{enumerate}
Thus the result follows.
\end{proof}
\subsection{Monomial polyhedrons}
For $d \geq 2,$ a $d$-tuple $\bl \alpha = (\alpha_1,\ldots,\alpha_d) \in \mb Q^d$ of rational numbers and a $d$-tuple of complex numbers $\bl z = (z_1,\ldots,z_d) \in \mb C^d$, we denote $\bl z^{\bl \alpha} := \displaystyle\prod_{k=1}^d z_k^{\alpha_k}.$ Consider a matrix $B \in M_d(\mb Q).$  We enumerate the row vectors of $B$ by $\mathcal F = \{\bl b^1,\ldots,\bl b^d\},$ where $\bl b^k = (b_1^k,\ldots,b_d^k).$ The monomial polyhedron associated to $B$ is defined by $$\mathscr U = \{ \bl z \in \mb C^d : |{\bl z}^{\bl b^k}| <1 \text{~for all ~} 1\leq k \leq d\},$$ unless for some $1 \leq k,j \leq d$, the quantity $ z_j^{b_j^k}$ is not defined due to the division of zero \cite[Equation 1.1]{bender2020lpregularity}. 

Without loss of generality, we  assume that $B \in M_d(\mb Z),~ \det(B)>0$ and $B^{-1} \succeq 0$ \cite[Equation 3.3]{bender2020lpregularity}. Set  $A = \adj B.$ The Smith Normal form of the matrix $A$ is given by $A = PDQ,$ where $P,Q \in GL_d(\mb Z)$ and $D ={\rm diag}(\delta_1,\ldots,\delta_d) \in M_d(\mb Z).$ Then $\m U$ is biholomorphically equivalent to $\mb D^d_{L(B)}/G,$ where $\mb D^d_{L(B)}$ is the product of some copies of the unit disc with some copies of the punctured unit disc and $G$ is isomorphic to the direct product of cyclic groups $\prod_{i=1}^d \mb Z / {\delta i} \mb Z$ \cite{bender2020lpregularity}. It implies that every irreducible representation of $G$ is one-dimensional.

 A point $\bl w_0$ belongs to the minimum boundary of $\mathscr U$ if and only $\bl w_0$ is an isolated point of the variety $V_0 = \{\bl w : {\bl w}^{\bl b^{k_i}} = {\bl w_0}^{\bl b^{k_i}} \text{~for all ~} 1\leq i \leq d_0 \}$ for the indices $k_1,\ldots,k_{d_0}$ such that $|{\bl w_0}^{\bl b^{k_i}}| =1.$  The Shilov boundary $\del_S \mathscr U$ of the monomial polyhedron $\mathscr U$ is the closure of the minimum boundary of $\mathscr U$ \cite[p. 1348]{MR155010}. 

\begin{cor}
Suppose that $u, v \in h_{G}^\infty(\mathscr U)$ are continuous on $\mathscr U \cup W$ for some relatively open subset $W \text{~of~} \del_S{\mathscr U}.$ If $T_uT_v=0$ on $\mb A_{\omega_\varrho}^2(\mathscr U),\,\, \varrho \in \widehat{G},$ then either $u=0$ or $v=0.$
\end{cor}
This is a direct application of Theorem \ref{result}.


\section{Generalized zero-product problem on the weighted Bergman space}\label{genzer}

In this section, using the results from \cite{MR4295248} and following the methods in \cite{MR2262785}, we prove Theorem \ref{expoly}. For $\alpha > -1,$ the continuous function $\omega_\alpha: \mb D \to (0,\infty)$ is defined by $\omega_\alpha(z) = (\alpha +1 ) (1-|z|^2)^\alpha.$ 
The weighted Bergman kernel $K^{(\alpha)}$ of $\mb A^2_{\omega_\alpha}(\mb D)$ is given by \Bea  K^{(\alpha)} (z,w) = \frac{1}{(1-z\ov{w})^{\alpha +2}}, \,\, z,w \in \mb D. \Eea 

We recall the Berezin type operators from \cite[p. 29]{MR1758653}. For $\alpha >-1,$ the Berezin type operator $B_\alpha,$ defined on $L^1(\mb D,\omega_\alpha dV),$ is given by 
\Bea
B_\alpha f(z) = \int_{\mb D} f \circ \phi_z(w) \omega_\alpha(w) dV(w),
\Eea
where $\phi_z(w)= \frac{w-z}{1-\ov{z}w}.$
However, we use the following expression of Berezin type operators more often (which is obtained after a change of variable):
\bea\label{berezin2}
B_\alpha f(z) = (\alpha +1)\int_{\mb D} \frac{(1-|z|^2)^{\alpha +2} (1-|w|^2)^\alpha}{|1-z\ov{w}|^{4+2\alpha}} f (w) dV(w).
\eea

We collect two results from \cite{MR4295248} which will be needed.
In the following $\widetilde{\Delta} = (1-|z|^2)^2 \frac{\del^2}{\del z \del \ov{z}},$ is the invariant Laplacian of $\mb D.$ 

\begin{thm}\label{exone}

\begin{enumerate}

\item Let $B_{\bl \alpha}(v) = f \overline{g}$ where $f$ and $g$ are holomorphic functions on $\mb D$ and $v \in L^1(\mb D, \omega_\alpha dV).$ Then either $f$ or $g$ is a constant.

\item Suppose that $\alpha$ is a non-negative integer. Let $f$ and $g$ be bounded harmonic functions on $\mb D$ and $h$ be a bounded $C^{2\alpha}$ function on $\mb D$ such that $\widetilde{\Delta}h,\ldots,\widetilde{\Delta}^{\alpha}h \in L^1(\mb D, \omega_\alpha dV)$. If $T_fT_g = T_h$ on $\mb A_{\omega_\alpha}^2(\mb D)$ then either $f$ is co-analytic or $g$ is analytic.
\end{enumerate}
\end{thm} 

\begin{proof}
See \cite[Theorem 10]{MR4295248} and its proof.
\end{proof}

 
\subsection{Weighted Bergman spaces on the polydisc} 
Consider $d>1.$ Let us denote $\bl \alpha = (\alpha_1,\ldots,\alpha_d),$ where $\alpha_i$'s are non-negative integers. The continuous function $\omega_{\bl \alpha}: \mb D^d \to (0,\infty)$ is defined by $\omega_{\bl \alpha}(\bl z) = \prod_{i=1}^d(\alpha_i +1 ) (1-|z_i|^2)^{\alpha_i}.$ The weighted Bergman space $\mb A^2_{\omega_{\bl \alpha}}(\mb D^d)$ is a reproducing kernel Hilbert space with the reproducing kernel $$\bl K^{(\bl \alpha)}(\bl z, \bl w) = \prod_{i=1}^d\frac{1}{(1-z_i\ov{w}_i)^{\alpha_i +2}}.$$ Before proving Theorem \ref{expoly}, we recall a few relevant notions.

For simplicity, we denote the weighted Bergman projection by $P_{\bl \alpha} : L^2_{\omega_{\bl \alpha}}(\mb D^d) \to  \mb A^2_{\omega_{\bl \alpha}}(\mb D^d).$ For $g \in \mb A^2_{\omega_{\bl \alpha}}(\mb D^d)$ and ${\bl w} \in \mb D^d,$ we have: $$P_{\bl \alpha}(\ov{g}\bl K^{(\bl \alpha)}_{\bl w})({\bl z}) = \inner{\ov{g}\bl K^{(\bl \alpha)}_{\bl w}}{\bl K^{(\bl \alpha)}_{\bl z}} = \inner{\bl K^{(\bl \alpha)}_{\bl w}}{g\bl K^{(\bl \alpha)}_{\bl z}} = \ov{g}({\bl w})\bl K^{(\bl \alpha)}_{\bl w}({\bl z}).$$. Hence \bea\label{bar} P_{\bl \alpha}(\ov{g}\bl K^{(\bl \alpha)}_{\bl w}) = \ov{g}({\bl w})\bl K^{(\bl \alpha)}_{\bl w} \, \text{and} \,P_{\bl \alpha}(g\bl K^{(\bl \alpha)}_{\bl w}) = g \bl K^{(\bl \alpha)}_{\bl w}.\eea

For $f \in L^1(\mb D^d,\omega_{\bl \alpha} dV),$ the Berezin type operators are given by 
\bea
B_{\bl \alpha} f(\bl z) = \int_{\mb D^d} \prod_{i=1}^d\left((\alpha_i +1)\frac{(1-|z_i|^2)^{\alpha_i +2} (1-|w_i|^2)^{\alpha_i}}{|1-z_i\ov{w_i}|^{4+2\alpha_i}} \right) f (\bl w) dV(\bl w).
\eea In addition to it, if $f \in \mb A^2_{\omega_{\bl \alpha}}(\mb D^d),$ then $B_{\bl \alpha} f(\bl z) = \prod_{i=1}^d(1-|z_i|^2)^{\alpha_i+2} \inner{f\bl K^{(\bl \alpha)}_{\bl z}}{\bl K^{(\bl \alpha)}_{\bl z}} = \prod_{i=1}^d(1-|z_i|^2)^{\alpha_i+2}f(\bl z)\bl K^{(\bl \alpha)}(\bl z,\bl z) =f(\bl z).$ Similarly one gets $B_{\bl \alpha} \ov{f} = \ov{f}.$

\begin{lem}\label{pr}
Suppose that $f = f_1 + \ov{f_2},$ $g=g_1 + \ov{g_2}$ are bounded harmonic functions with $f_i,g_i$ holomorphic and $h$ is bounded in $\mb D^d.$ Then the following are equivalent.
\begin{enumerate}
    \item $T_fT_g = T_h$
    \item For all ${\bl z},{\bl w} \in \mb D^d,$ \Bea && f_1({\bl z})g_1({\bl z}) + \ov{f_2}(\ov{{\bl w}})\ov{g_2}(\ov{{\bl w}})+f_1({\bl z})\ov{g_2}(\ov{{\bl w}})\\&=&  \int_{\mb D^d} \prod_{i=1}^d \left(\frac{(1-z_iw_i)^{\alpha_i +2 }}{(1-z_i\ov{\eta_i})^{\alpha_i+2}(1-w_i\eta_i)^{\alpha_i+2}}\right) (h(\bl \eta)-\ov{f_2}(\bl \eta)g_1(\bl \eta))\omega_{\bl \alpha}(\bl \eta) dV(\bl \eta).\Eea
\item For all ${\bl z} \in \mb D^d,$ $$f_1({\bl z})g_1({\bl z}) + \ov{f_2}({\bl z})\ov{g_2}({\bl z})+f_1({\bl z})\ov{g_2}({\bl z}) = B_{\bl \alpha}(h-\ov{f_2}g_1)({\bl z}).$$ \end{enumerate} 
\end{lem}
\begin{proof}
 For ${\bl w} \in \mb D^d,$ we get from Equation \eqref{bar} $$T_g\bl K^{(\bl \alpha)}_{\bl w} = P_{\bl \alpha}(g_1\bl K^{(\bl \alpha)}_{\bl w} + \ov{g_2}\bl K^{(\bl \alpha)}_{\bl w}) = g_1\bl K^{(\bl \alpha)}_{\bl w} + \ov{g_2}({\bl w}) \bl K^{(\bl \alpha)}_{\bl w}.$$ Then \Bea T_fT_g\bl K^{(\bl \alpha)}_{\bl w} &=& P_{\bl \alpha}((f_1+\ov{f_2})(g_1\bl K^{(\bl \alpha)}_{\bl w} + \ov{g_2}({\bl w}) \bl K^{(\bl \alpha)}_{\bl w}))\\&=&f_1g_1\bl K^{(\bl \alpha)}_{\bl w} + \ov{g_2}({\bl w})f_1\bl K^{(\bl \alpha)}_{\bl w} + \ov{f_2}({\bl w})\ov{g_2}({\bl w})\bl K^{(\bl \alpha)}_{\bl w} +P_{\bl \alpha}(\ov{f_2}g_1\bl K^{(\bl \alpha)}_{\bl w}).\Eea
 Moreover, $T_fT_g=T_h$ if and only if $T_fT_g\bl K^{(\bl \alpha)}_{\bl w}=T_h\bl K^{(\bl \alpha)}_{\bl w}$ for all ${\bl w} \in \mb D^d,$ equivalently, \Bea  f_1g_1\bl K^{(\bl \alpha)}_{\bl w} + \ov{g_2}({\bl w})f_1\bl K^{(\bl \alpha)}_{\bl w} + \ov{f_2}({\bl w})\ov{g_2}({\bl w})\bl K^{(\bl \alpha)}_{\bl w} +P_{\bl \alpha}(\ov{f_2}g_1\bl K^{(\bl \alpha)}_{\bl w}) &=& P_{\bl \alpha}(h\bl K^{(\bl \alpha)}_{\bl w}). \Eea
 That is, for ${\bl z} \in \mb D^d,$
 \Bea &&f_1({\bl z})g_1({\bl z}) + \ov{g_2}({\bl w})f_1({\bl z}) + \ov{f_2}({\bl w})\ov{g_2}({\bl w})\\ &=& \frac{1}{\bl K^{(\bl \alpha)}_{\bl w}({\bl z})}P_{\bl \alpha}(h\bl K^{(\bl \alpha)}_{\bl w} - \ov{f_2}g_1\bl K^{(\bl \alpha)}_{\bl w})({\bl z}),
  \Eea
  which is \emph{2.} with ${\bl w}$ replaced by $\ov{{\bl w}}.$ This proves that \emph{1.} and \emph{2.} are equivalent.
  
  Now replacing $\bl w$ with $\ov{\bl z},$ we get \Bea && f_1(\bl z)g_1(\bl z) + \ov{g_2}(\bl z)f_1(\bl z) + \ov{f_2}(\bl z)\ov{g_2}(\bl z) \\ &=&  \int_{\mb D^d} \prod_{i=1}^d\left((\alpha_i +1)\frac{(1-|z_i|^2)^{\alpha_i +2} (1-|\eta_i|^2)^{\alpha_i}}{|1-z_i\ov{\eta_i}|^{4+2\alpha_i}} \right) (h-\ov{f_2}g_1) (\bl \eta) dV(\bl \eta) \\ &=& B_{\bl \alpha}(h-\ov{f_2}g_1)(\bl z), \Eea
  which provides the expression in \emph{3.} Note that in \emph{2.} both sides are holomorphic on $\mb D^{d} \times \mb D^d.$ We assume \emph{3.} It holds on the subset $\{(\bl z,\bl w) \in \mb D^{d} \times \mb D^d : \bl z = \ov{\bl w} \}$ and thus holds on $\mb D^{d} \times \mathbb D^d.$ Thus, \emph{3.} implies \emph{2.}. This argument is analogous to \cite[p. 205]{MR1867348}.\end{proof}
\begin{proof}[Proof of Theorem \ref{expoly}]
For bounded pluriharmonic functions $f$ and $g$ on $\mb D^d$, we can write $f=f_1+\ov{f}_2$ and $g=g_1+\ov{g}_2$ for holomorphic functions $f_i, g_i,\,i=1,2$. For $\bl a \in \mb D^d,$ we note that $T_{f_1+\ov{f}_2}\bl K^{(\bl \alpha)}(\cdot, \bl a) = P_{\bl \alpha}((f_1+\ov{f}_2)\bl K^{(\bl \alpha)}(\cdot, \bl a)) = (f_1+\ov{f}_2(\bl a))\bl K^{(\bl \alpha)}(\cdot, \bl a),$ which implies, $$T_{f_1+\ov{f}_2}T_{g_1+\ov{g}_2}\bl K^{(\bl \alpha)}(\cdot, \bl a) = \big (f_1g_1 + f_1\ov{g}_2(\bl a) + \ov{f}_2(\bl a)\ov{g}_2(\bl a) \big )\bl K^{(\bl \alpha)}(\cdot, \bl a)$$ $$ +  P_{\bl \alpha}(g_1\ov{f}_2 \bl K^{(\bl \alpha)}(\cdot, \bl a)).$$ 
From Lemma \ref{pr}, we get, $$B_{\bl \alpha}(v) = f_1\ov{g_2}$$ for $$v= h-\ov{f_2}g_1- f_1g_1 - \ov{f_2}\ov{g_2},$$ since $B_{\bl \alpha}f_1g_1=f_1g_1$ and $B_{\bl \alpha}\ov{f_2}\ov{g_2}=\ov{f_2}\ov{g_2}$ $.$ For a fixed $w_0 \in \mb D,$ let $(w_0, \bl 0)$ be the point $(w_0,0, \cdots 0) \in \mathbb D^d.$ Then we have,
\Bea (B_{\bl \alpha}v)(w_0,\bl 0) &=& f_1(w_0,\bl 0)\ov{g_2}(w_0,\bl 0)\\&=& (\alpha_1+1)\int_{\mb D}  \frac{(1-|w_0|^2)^{\alpha_1 +2}}{|1-w_0\ov{w}|^{4+2\alpha_1}} u(w) (1-|w|^2)^{\alpha_1} dV_1(w)\\&=& B_{\alpha_1}u(w_0), \Eea where $$u(\eta) = \prod_{i=2}^d (\alpha_i +1) \int_{\mb D^{d-1}} v(\eta,\bl z) (1-|z_2|^2)^{\alpha_2}\cdots(1-|z_d|^2)^{\alpha_d} dV_{d-1}(\bl z)$$ for $\eta \in \mb D.$ A similar argument as in \cite[p. 304, Theorem 4]{MR2262785} shows that $u$ is in $L^1(\mb D, \omega_{\alpha_1}dV_1).$ Then using Theorem \ref{exone}, we conclude that either $\del_1f(w_0, \bl 0) = 0$ or $\del_1\ov{g}(w_0, \bl 0) =0.$ To complete the proof we argue as in \cite[p. 304, Theorem 4]{MR2262785}, using the automorphisms of $\mathbb D^d.$ For $\bl z = (z_1, z_2, \cdots z_d) \in \mathbb D^d,$ let $\varphi_{\bl z}$ be the automorphism given by
$$\varphi_{\bl z} (\bl w) = \prod_{i=1}^d~ \frac{z_j-w_j}{1-\ov{z_j} w_j}.$$ 
Then, for $f, g, h \in L^\infty(\mathbb D^d)$ we have $T_f T_g = T_h$ implies $$T_{f \circ \varphi_{\bl z}} T_{g \circ \varphi_{\bl z}} = T_{h \circ \varphi_{\bl z}}. $$ The above can be proved, first for the weighted Bergman space $\mathbb A^2_{\omega_\beta}(\mathbb D)$ on $\mb D,$ where $\beta$ is a non-negative integer
following the proof of Lemma 8 in \cite{MR1079815}, using the unitary operator $V_\psi f = f \circ \psi~ \psi^{\frac{\beta}{2}+1}$ on $L^2 (\mathbb D, \omega_\beta dV_1).$ It then readily extends to $\mathbb A^2_{\omega_\alpha}(\mathbb D^d).$ This finishes the proof.
\end{proof}

\subsection{Weighted Bergman spaces on the symmetrized polydisc}
Let $\alpha$ be a non-negative integer. For $\bl \alpha= (\alpha,\ldots,\alpha),$ the continuous function $\omega_{\bl \alpha}$ is $\mathfrak{S}_d$-invariant.
Consider the probability measure $$\omega_{\bl \alpha}(\bl z)dV(\bl z)=\big(\frac{\alpha+1}{\pi}\big)^d\Big(\prod_{i=1}^d(1-r_i^2)^{\alpha}r_idr_id\theta_i\Big)$$ on the polydisc $\mb D^d.$ Let $dV^{(\alpha)}_{\bl s}$ be the measure on the symmetrized polydisc $\mb G_d$  obtained by the change of variables formula  \cite[p. 106]{MR733691}:
\bea\label{cov}
\int_{\mb G_d}f dV^{(\alpha)}_{\bl s}=\frac{1}{d!}\int_{\mb D^d}(f\circ\bl s)\vert J_{\bl s}\vert^2\omega_{\bl \alpha}dV,
\eea
where $ J_{\bl s}(\bl z)$ is the complex jacobian of the symmetrization map $\bl s.$ The weighted Bergman space $ \mb{A}_{\omega_{\bl \alpha}} (\mb G_d),$ on the  symmetrized polydisc $\mb G_d$ is the subspace of $L^2(\mb G_d, dV^{(\alpha)}_{\bl s})$ consisting of holomorphic functions. 
The weighted Bergman space $\mb{A}_{\widetilde{\omega}_{\bl \alpha}} (\mb G_d)$ is isometrically isomorphic to $\mb P_{\rm anti}(\mb A^2_{\omega_{\bl \alpha}}(\mb D^d))$ for $\omega_{\bl \alpha}=\widetilde{\omega}_{\bl \alpha} \circ \bl s$ \cite[p. 2363]{MR3043017}.
 Then from Theorem \ref{main1}, we deduce that for $\widetilde{u} = u \circ \bl s,\widetilde{v}=v \circ \bl s$ and $\widetilde{q} = q \circ \bl s$ symmetric symbols in $L^\infty(\mb D^d),$ $T_{\widetilde{u}}T_{\widetilde{v}} = T_{\widetilde{q}}$ on $\mb A^2_{\omega_{\bl \alpha}}(\mb D^d)$ if and only if $T_uT_v = T_q$ on  $\mb{A}_{\widetilde{\omega}_{\bl \alpha}} (\mb G_d).$ Then the proof of Theorem \ref{forsympoly} is straightforward from Theorem \ref{expoly}.

\section*{Conflict of interest} The authors declare there is no conflict of interest.


\bibliographystyle{siam}
\bibliography{Bibliography.bib}

\end{document}